\documentclass[12pt]{gAPA2e}
\usepackage{subfigure}

\theoremstyle{plain}
\newtheorem{theorem}{Theorem}[section]
\newtheorem{corollary}[theorem]{Corollary}
\newtheorem{lemma}[theorem]{Lemma}
\newtheorem{proposition}[theorem]{Proposition}

\theoremstyle{remark}
\newtheorem{remark}{Remark}

\theoremstyle{definition}

\usepackage{MnSymbol}
\usepackage{bbm}
\usepackage{amsbsy}
\usepackage{graphicx}
\usepackage{color}

\newtheorem{assumption}{Assumption}
\DeclareMathOperator{\grad}{\nabla}

\DeclareMathOperator{\ddiv}{\nabla\cdot}
\newcommand{\bfu}{\textbf{\textit{u}}}
\newcommand{\bfv}{\textbf{{v}}}
\newcommand{\bfw}{\textbf{{w}}}
\newcommand{\bfx}{\textbf{{x}}}
\newcommand{\bfz}{\textbf{{z}}}

\newcommand{\bfphi}{\boldsymbol{\phi}}

\newcommand{\bfnu}{\boldsymbol{\nu}}
\newcommand{\Hdiv}{H({\rm{}div};\Omega_R)}

\usepackage{stmaryrd}

\newcommand{\jmp}[1]{\lsem #1 \rsem }   
\newcommand{\avg}[1]{\left\{\!\!\left\{#1\right\}\!\!\right\}}

\newcommand{\pmr}[1]{\textcolor{black}{#1}}
\begin{document}

\title{A Trefftz Discontinuous Galerkin Method for Time Harmonic Waves with a Generalized Impedance Boundary Condition}

\author{Shelvean Kapita$^{\rm a}$$^{\ast}$\thanks{$^\ast$Corresponding
author. Current address: Department of Mathematics, University of Georgia, 220 DW Brooks Dr, Athens, GA 30602. Email: shelvean.kapita@uga.edu \vspace{6pt}}, Peter Monk$^{\rm b}$ and Virginia Selgas$^{\rm c}$\\\vspace{6pt} $^{a}${\em{Institute for Mathematics and its Applications, College of Science and Engineering, 306 Lind Hall,
Minneapolis, MN USA 55455}}; $^{b}${\em{Department of Mathematical Sciences, University of Delaware, Newark DE 19716, USA}}; $^{c}${\em{Departamento de Matem\'{a}ticas, Universidad de Oviedo,
EPIG, 33203 Gij\'{o}n, Spain}}}

\maketitle

\begin{abstract}
We show how a Trefftz Discontinuous Galerkin (TDG) method for the
displacement form of the Helmholtz equation can be used to approximate
problems having a generalized impedance boundary condition (GIBC)
involving surface derivatives of the solution.  Such boundary
conditions arise naturally when modeling scattering from a scatterer
with a thin coating.  The thin coating can then be approximated by a
GIBC.  A second place GIBCs arise is as higher order absorbing
boundary conditions.  This paper also covers both cases.  Because the
TDG scheme has discontinuous elements, we propose to couple it to a
surface discretization of the GIBC using continuous finite elements.
We prove convergence of the resulting scheme and demonstrate it with
two numerical examples.

\begin{keywords} Helmholtz equation, Trefftz,  Discontinuous Galerkin,  Generalized Impedance Boundary Condition, Error estimate, Artificial Boundary Condition
\end{keywords}

\begin{classcode}\end{classcode}
\end{abstract}


\section{Introduction}
\pmr{The Trefftz method, in which a linear combination of simple solutions of the underlying partial differential equation 
on the whole solution domain are
used to approximate the solution of the desired problem, dates back to the 1926 paper of Trefftz~\cite{Trefftz}.  A  historical discussion in relation to Ritz and Galerkin methods can be found in ~\cite{Gander}.  From our point of view, a key paper in this area is that of Cessenat and D\'espres~\cite{cessenat03}
who analyzed the use of a local Trefftz space on a finite element grid to approximate the solution of the Helmholtz equation~\cite{cessenat03}.  This was later shown to be a special case of the Trefftz Discontinuous Galerkin (TDG) method~\cite{buf07,git07} which opened the
way for a more general error analysis. For more recent work in which boundary integral operators are used to contruct the Trefftz space, see for example \cite{bar,hof}.  The aforementioned work all concerns the standard pressure field formulation of acoustics
which results in a scalar Helmholtz equation. Indeed, TDG methods are well developed for
the Helmholtz, Maxwell and Navier equations with standard boundary
conditions and a recent survey can be found in~\cite{tsurvey}. For the displacement form, a TDG method has been proposed by
Gabard~\cite{gabard} also using simple boundary conditions.
}

\pmr{Because of the unusual boundary conditions considered in this paper, we propose to use the displacement form of the Trefftz Discontinuous Galerkin (TDG) method for
approximating  solutions of the Helmholtz equation
governing scattering of an acoustic wave (or suitably polarized
electromagnetic wave) by a bounded object.  This is because the scatterer is assumed
to be modeled by a Generalized Impedance Boundary condition (GIBC).
These boundary conditions arise as approximate asymptotic models of
thin coatings or gratings
(\cite{EngquistNedelec,BCH,BCHprevious,Vernhet}). 
Importantly, they also arise as approximate absorbing boundary conditions (ABCs) and our paper shows how to
handle these boundary conditions.  As
  far as we are aware, the  displacement TDG method has not been  analyzed to date.  We provide such an analysis and this is a one of the 
  contributions of our paper.}

In order to define the problem under consideration more precisely, let
$D\subset \mathbb{R}^2$ denote the region occupied by the
scatterer.  We assume that $D$ is
an open bounded domain with connected complement having a smooth
boundary $\Gamma =\partial D$. Then we can define $\nabla _{\Gamma} $
to be the surface gradient and $\nabla_{\Gamma}\cdot$ to be the
surface divergence on $\Gamma$ (see for example
\cite{ColtonKress}). In addition $\bfnu$ denotes the outward unit
normal on $\Gamma$.

Let $k\in\mathbb{R}$, $k\neq 0$, denote the wave number of the field,
and suppose that a given incident field $u^i$ impinges on the
scatterer. We want to approximate the scattered field $u\in
H^1_{\rm{}loc}(\mathbb{R}^2\setminus\overline{D})$ that is the
solution of
\begin{equation}\label{eq-fwdprob}
\begin{array}{rcl}
\displaystyle\Delta u+k^2u&=&0 \quad \mbox{ in }\Omega=\mathbb{R}^2\setminus \overline{D} \, ,\\[2ex]
\displaystyle\nabla_{\Gamma } \cdot \left(\beta \nabla_{\Gamma } u\right)+\frac{\partial u}{\partial \boldsymbol{ \nu } } +\lambda u
&=&-g \quad \mbox{ on }\Gamma :=\partial D \, ,\\[2ex]
\displaystyle\lim _{r:=|\bfx|\to\infty} \sqrt{r} \big( \frac{\partial u}{\partial r}-iku\big) &=& 0 \, .
\end{array} 
\end{equation}
The last equation,  the Sommerfeld radiation condition (SRC), holds uniformly in $\hat{\bfx}=\bfx/r$. In addition
\[
g=\nabla_{\Gamma } \cdot \left(\beta \nabla_{\Gamma } u^i \right)+\frac{\partial u^i}{\partial \bfnu }+\lambda u^i \, ,
\]
where $u^i$ is the given incident field and is assumed
to be a smooth solution of the Helmholtz equation $\Delta
u^i+k^2u^i=0$ in a neighborhood of $D$.  For example, if $u^i$ is a
plane wave then $u^{i}=\exp(ik\mathbf{x}\cdot\mathbf{d})$, where
$\mathbf{d}$ is the direction of propagation of the plane wave
and $\vert \mathbf{d}\vert=1$.  Alternatively $u^i$
could be the field due to a point source in $\mathbb{R}^2\setminus
\overline{D}$.  The coefficient functions $\beta$ and $\lambda$ are
used to model the thin coating on $\Gamma$ and we shall give details
of the assumptions on these coefficients in the next section.

As can be seen from the second equation in (\ref{eq-fwdprob}), the
GIBC involves a non-homogeneous second order partial differential
equation on the boundary of the scatterer, and this complicates the
implementation using a TDG method which uses discontinuous local
solutions of the homogeneous equation element by
element.  In addition, because the problem is posed on an infinite
domain we need to truncate the domain to apply the TDG method, and
then apply a suitable artificial boundary condition (ABC) on the outer
boundary.  \pmr{Because TDG methods have discontinuous basis functions, when the GIBC or ABC involve derivatives, these boundary conditions} can be applied more easily if we convert
them to displacement based
equations, so we propose to solve (\ref{eq-fwdprob})
by converting it to a vector problem.  To this end, we introduce
$\bfv=\nabla u$, in which case $\nabla\cdot\bfv=\Delta u=-k^2u$.
Using this relationship we see that $\bfv$ should satisfy
\begin{equation}\label{vbasic}
\begin{array}{rcl}
\nabla \nabla\cdot\bfv+k^2\bfv&=&0\mbox{ in }\mathbb{R}^2\setminus\overline{D}\, ,\\[2ex]
\displaystyle\nabla_{\Gamma } \cdot \left(\beta \nabla_{\Gamma } (\nabla\cdot\bfv)\right)-{k^2}\bfv\cdot\boldsymbol{\nu } +\lambda \nabla\cdot\bfv
&=&\displaystyle{k^2}g \quad \mbox{ on }\Gamma :=\partial D \, ,\\[2ex]
\displaystyle\lim _{r:=|\bfx|\to\infty} \sqrt{r} \big( \nabla\cdot\bfv-ik\bfv\cdot\hat{\mathbf{x}}\big) &=& 0  \, ,
\end{array} 
\end{equation}
where the radiation condition (last equation) holds uniformly for  all in directions $\hat{\mathbf{x}}:=\mathbf{x}/\vert\mathbf{x}\vert$.

The use of the displacement variable for the Helmholtz equation with standard boundary conditions in the context of plane wave methods was considered by Gabard in \cite{gabard}, but no error estimates were proved. In particular he used the PUFEM \cite{PUFEM} and DEM \cite{DEM} approaches, not TDG.
The use of the displacement vector as the primary variable is often necessary in studies of fluid-structure interaction (see e.g. \cite{wang}). To date, no error estimates have been proved for the displacement based formulation with or without the GIBC.  The vector formulation is useful in its own right.  For example, using finite element methods, Brenner et al. \cite{brenner} show that a vector
formulation can also be advantageous for sign changing materials, although we do not consider that problem here.

Our approach to discretizing (\ref{vbasic}) is to use TDG in a bounded
subdomain of $\mathbb{R}^2\setminus \overline{D}$, and standard finite
elements or trigonometric polynomial based methods to discretize the
GIBC on the boundary.  The domain is truncated using the
Neumann-to-Dirichlet (NtD) map on an artificial boundary that is taken
to be a circle.  Other truncation conditions could be used.  Since it
is not the focus of the paper, we assume for simplicity that the NtD
map is computed exactly.  The discretization of the NtD map could be
analyzed using the techniques from \cite{shelvean_phd,shelvean_paper},
and it is also possible to use an integral equation approach to
approximate the NtD on a more general artificial boundary but this
remains to be analyzed.

Our analysis of the discrete problem follows the pattern of the
analysis of finite element methods for approximating the standard
problem of scattering by an impenetrable scatterer using the
Dirichlet-to-Neumann boundary condition from \cite{koyama}. We first
show that the GIBC can be discretized leaving the displacement
equation continuous.  Then we show that this semi-discrete problem can
also be discretized successfully.  The analysis of the error in the
TDG part of the problem is motivated by the analysis of TDG for
Maxwell's equations in \cite{HMP_Maxwell} and uses the Helmholtz
decomposition of the vector field satisfying (\ref{vbasic}) as a
critical tool.

The contributions of this paper are 1) a first application and
analysis of TDG to the displacement Helmholtz problem; 2) a method for
incorporating a discretization of the GIBC into the TDG scheme using
novel numerical fluxes from \cite{shelvean_phd}; 3) an error analysis
of the fully discrete problem (except for the NtD map as described
earlier), and the first numerical results for TDG applied to this
problem.

In the remainder of the paper we use bold font to represent vector
fields and we will work in $\mathbb{R}^2$. We utilize the usual
gradient and divergence operators (both in the domain and on the
boundary), and also a vector and scalar curl defined by
\[
{\rm\bf{} curl}\;v=\left(\!\!\begin{array}{r}
\frac{\partial v\,}{\partial x_2} \\[1ex]
-\frac{\partial v\,}{\partial x_1}
\end{array}\!\!\right)
\quad \mbox{ and } \quad
{\rm curl}\;\bfv=\frac{\partial v_2}{\partial x_1}-\frac{\partial v_1}{\partial x_2} \, ,
\]
for any $v:\mathbb{R}^2\to\mathbb{C}$ and  $\bfv:\mathbb{R}^2\to\mathbb{C}^2$. 

The paper proceeds as follows.  In the next section we formulate
problem (\ref{vbasic}) in a variational way and show it is well posed
using the theory of Buffa~\cite{Buffa2005}.  Then in Section
\ref{semi} we describe and analyze the discretization of the GIBC
using finite elements (or trigonometric basis functions).  The fully
discrete TDG scheme is described in Section \ref{trefftz} where we
also prove a basic error estimate and show well-posedness of the fully
discrete problem.  We then prove convergence in a special mesh
independent norm.  In Section~\ref{num} we provide a preliminary numerical test
of the algorithm, and in Section~\ref{concl} we draw some conclusions.

\section{Variational Formulation of the Displacement Method}
In this section we give details of our assumptions on the coefficients
in the GIBC, and formulate the displacement problem (\ref{vbasic}) in
variational setting suitable for analysis. Then we show that the
problem is well-posed.  The functions $\beta ,\lambda \in L^{\infty}
(\Gamma )$ in (\ref{eq-fwdprob}) are complex valued functions and we
assume that there exists a constant $c>0$ such that
\begin{equation}\label{hyp-E!forwardprob}
\Re(\beta)\geq c ,\, \Im(\beta)\leq 0 \mbox{ and } \Im(\lambda)\geq 0\quad\mbox{ a.e. on }\Gamma.
\end{equation}
Of key importance will be the operator $G_\Gamma:H^{-1}(\Gamma)\to
H^{1}(\Gamma)$ defined weakly as the solution operator for the
boundary condition on $\Gamma$ relating the Neumann and Dirichlet
boundary data there. More precisely, for each $\eta\in H^{-1}(\Gamma)$
we define $G_{\Gamma}\eta\in H^1(\Gamma)$ to be the solution of
\begin{equation}
\int_{\Gamma } (\beta  \,\nabla_{\Gamma  }  (G_{\Gamma\!}\eta) \cdot  \nabla_{\Gamma  } \overline{\xi }-\lambda \, G_{\Gamma\!}\eta \, \overline{\xi })\,dS 
=
\int_{\Gamma } \! \eta \,\overline{\xi}\,dS
\qquad \forall \xi\in H^1(\Gamma ) \, .
\label{Gdef}
\end{equation}
An essential assumption is the following.
\begin{assumption}\label{A}
The only solution $u\in H^1(\Gamma)$ of 
\[
\displaystyle\nabla_{\Gamma } \cdot \left(\beta \nabla_{\Gamma } u\right) +\lambda u
=0
\]
is $u=0$. 
\end{assumption}
We will show that Assumption \ref{A} together with the conditions (\ref{hyp-E!forwardprob}) ensure that the operator $G_{\Gamma}: H^{-1}(\Gamma )\to H^1(\Gamma )$ is well-defined.

\begin{remark} One possible condition under which Assumption~\ref{A} holds is, 
\[
\mbox{Either }\Im(\beta)\leq -c< 0 \mbox{ or }\Im(\lambda)\geq c>0\quad \mbox{ a.e. on a segment }\Lambda\subset D.
\]
\end{remark}

\begin{remark}
On the one hand, the assumptions in (\ref{hyp-E!forwardprob}) concerning the imaginary
parts of $\beta$ and $\lambda$ are governed by physics, since these
quantities represent absorption when our model is deduced as an
approximation of the Engquist-N\'ed\'elec condition modeling the
diffraction of a time-harmonic electromagnetic wave by a perfectly conducting  object covered by a thin dielectric layer (see
\cite{EngquistNedelec}).

On the other hand,  the hypothesis in (\ref{hyp-E!forwardprob}) on the real part of $\beta$ is technical and
ensures ellipticity (see \cite[Th.2.1]{BCHprevious}); however, this property is fulfilled in the example of a medium with a thin coating (see \cite{EngquistNedelec}).  It would also be possible to allow
$\Re(\beta)\leq -c<0$ on $\Gamma$ as might be encountered modeling
meta-materials, but a sign changing coefficient would require a more
elaborate study.

The role of these properties will be clarified in
Lemma \ref{lemma-Gok}.
\end{remark}

The assumptions on the coefficients in (\ref{hyp-E!forwardprob}) together with Assumption~\ref{A} ensure that problem (\ref{eq-fwdprob}) has a unique weak solution $u$ in the space
$V = \{ v \in H^1_{loc} (\Omega ) \, ; \,\, v |_{\Gamma  } \in H^1 (\Gamma ) \}$ (see later and \cite[Th.2.1]{BCHprevious}).


To solve (\ref{vbasic}) we first truncate the domain. \pmr{We wish to analyze the error introduced in approximating a scattering problem, concentrating on the discretization of the GIBC, so we truncate the domain using a simple analytic Neumann-to-Dirichlet map.  Obviously other more general
truncation approaches such as integral equations could be used.  Indeed, in the numerical section, we shall consider a GIBC that arises from approximating the Neumann-to-Dirichlet map (a higher order ABC).}

Let $B_R$ denote
the ball of radius $R$ centered at the origin and set
$\Omega_R=B_R\setminus \overline{D}$ be our computational domain
(i.e. the bounded domain that we will mesh for the UWVF) and $ \Sigma
_R=\partial B_R$, where the radius $R$ is taken large enough to
enclose $\overline{D}$ (see Fig.~\ref{cartoon} for a diagram
illustrating the major geometric elements of the problem). The
following Neumann-to-Dirichlet (NtD) map $N_R:H^{-1/2}(\Sigma_R)\to
H^{1/2}(\Sigma_R)$ will provide the ABC on $\Sigma_R$. In particular
let $v\in H^1_{\rm{}loc}(\mathbb{R}^2\setminus \overline{B_R})$ solve
the exterior problem
\begin{equation}\label{eq-NR}
\begin{array}{rcll}
\displaystyle\Delta v+k^2v&=&0 &\quad \mbox{ in }\mathbb{R}^2\setminus \overline{B_R} \, ,\\[2ex]
\displaystyle\frac{\partial v}{\partial r}&=&f&\quad\mbox{ on }\Sigma_R\,,\\[2ex]
\displaystyle\lim _{r:=|\bfx|\to\infty} \sqrt{r} \big( \frac{\partial v}{\partial r}-ikv\big) &=& 0 & \quad \mbox{ uniformly in direction }\hat{\mathbf{x}}=\mathbf{x}/\vert\mathbf{x}\vert \, ,
\end{array} 
\end{equation}
for some $f\in H^{-1/2}(\Sigma_R)$, then $N_R(f)=v|_{\Sigma_R}$. Let
us recall that $N_R : H^{- 1/2} ( \Sigma _R ) \to H^{1/2} ( \Sigma _R
)$ is an isomorphism since its inverse, the Dirichlet-to-Neumann
  map, is also an isomorphism~\cite{ColtonKress}. Obviously, the
solution of (\ref{eq-fwdprob}) satisfies $u|_{\Sigma_R}=N_R(\partial
u/\partial r)$ and, in consequence, using the fact that
$\nabla\cdot\bfv=-k^2u$,
\[
(\nabla\cdot\bfv )|_{\Sigma _R} =-k^2N_R(\bfv\cdot\bfnu) \, ,
\]
where we denote by $\bfnu:=\mathbf{x}/R$ the outward unit normal on $\Sigma_R$.
In the same way, for the solution of (\ref{eq-fwdprob}) we have that
\[
(\nabla\cdot\bfv ) |_\Gamma=-k^2G_\Gamma(\bfv\cdot\bfnu+g) \, .
\]

Now we can write down a weak form for the boundary value problem (\ref{vbasic}) in the usual way, multiplying the first equation in (\ref{vbasic}) by a test vector function $\mathbf{w}$ and integrating by parts:
\begin{eqnarray*}
0&=&
\int_{\Omega_R}(\nabla \nabla\cdot\bfv+k^2\bfv)\cdot \overline{\mathbf{w}}\,d\bfx \\
&=&
\int_{\Omega_R}(-\nabla\cdot\bfv\,\nabla\cdot\overline{\mathbf{w}}+k^2\bfv\cdot \overline{\mathbf{w}})\,d\bfx
 +\int_{\Sigma_R}\nabla \cdot \bfv\,\bfnu\cdot\overline{\mathbf{w}}\,dS
 -\int_{\Gamma}\nabla \cdot \bfv\,\bfnu\cdot\overline{\mathbf{w}}\,dS \, ,
\end{eqnarray*}
where the minus sign in the last term is due to the normal field $\bfnu$ pointing outward $D$.  Using the NtD map $N_R$ and the boundary solution map
$G_\Gamma$, the above equation can be rewritten as the problem of
finding $\bfv\in H({\rm{}div};\Omega_R)$ such that
\begin{equation}\label{prob-FwdProbV}
\begin{array}{rcl}
&&\displaystyle\int_{\Omega_R}\! (\frac{1}{k^2}\,\nabla\cdot\bfv\,\overline{\nabla\cdot\bfw}-\bfv\cdot\overline{\bfw})\,d\bfx 
+\int_{\Sigma _R}\!\! N_R(\bfv\cdot\bfnu)\,\overline{\bfw\cdot\bfnu}\,dS\\[2ex]
&&\displaystyle\qquad- \int_{\Gamma }\!\! G_{\Gamma  \! }(\bfv\cdot\bfnu)\,\overline{\bfw\cdot\bfnu}\,dS=
 \int_{\Gamma } \!\! G_{\Gamma  }(g) \,\overline{\bfw\cdot\bfnu}\,dS
\, ,
\end{array}
\end{equation}
for any $\bfw\in H({\rm{}div};\Omega_R)$. It will be convenient to associate with the left hand side of (\ref{prob-FwdProbV}) the sesquilinear form $a: H({\rm{}div};\Omega_R) \times H({\rm{}div};\Omega_R) \to \mathbb{C}$ defined by
\begin{equation}
  a(\bfv,\mathbf{w})
=\displaystyle\int_{\Omega_R}\! (\frac{1}{k^2}\,\nabla\cdot\bfv\,\overline{\nabla\cdot\bfw}-\bfv\cdot\overline{\bfw})\,d\bfx 
  +\int_{\Sigma _R}\!\! N_R(\bfv\cdot\bfnu)\,\overline{\bfw\cdot\bfnu}\,dS- \int_{\Gamma }\!\! G_{\Gamma  \! }(\bfv\cdot\bfnu)\,\overline{\bfw\cdot\bfnu}\,dS.\label{adef}
\end{equation}

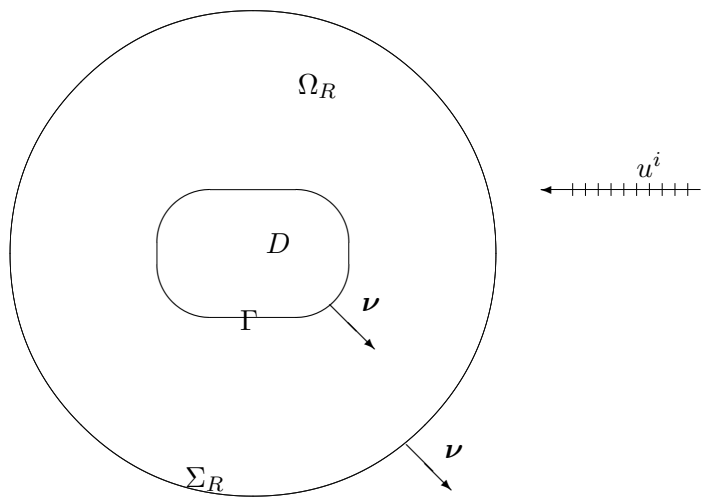
\begin{figure}
\begin{center}
\setlength{\unitlength}{17mm}
\begin{picture}(3,5)
  \linethickness{0.3pt}
  \put(-2.5,0.5){
  \qbezier(4.9,2.0)(4.9,2.787)(4.3435,3.3435)
  \qbezier(4.3435,3.3435)(3.787,3.9)(3.0,3.9)
  \qbezier(3.0,3.9)(2.213,3.9)(1.6565,3.3435)
  \qbezier(1.6565,3.3435)(1.1,2.787)(1.1,2.0)
  \qbezier(1.1,2.0)(1.1,1.213)(1.6565,0.6565)
  \qbezier(1.6565,0.6565)(2.213,0.1)(3.0,0.1)
  \qbezier(3.0,0.1)(3.787,0.1)(4.3435,0.6565)
  \qbezier(4.3435,0.6565)(4.9,1.213)(4.9,2.0)}
 \put(0.50,2.5){\oval(1.5,1.0)}
 \put(1.1,2.1){\vector(1,-1){0.35}}
 \put(1.7,1){\vector(1,-1){0.35}}
 \put(4,3){\vector(-1,0){1.25}}
 \put(3.80,2.95){\line(0,1){0.10}}
 \put(3.70,2.95){\line(0,1){0.10}}
 \put(3.60,2.95){\line(0,1){0.10}}
 \put(3.50,2.95){\line(0,1){0.10}}
 \put(3.40,2.95){\line(0,1){0.10}}
 \put(3.30,2.95){\line(0,1){0.10}}
 \put(3.20,2.95){\line(0,1){0.10}}
 \put(3.10,2.95){\line(0,1){0.10}}
 \put(3.00,2.95){\line(0,1){0.10}}
 \put(3.90,2.95){\line(0,1){0.10}}
 \put(1.350,2.05){$\boldsymbol{\nu }$}
 \put(2.0,0.9){$\boldsymbol{\nu }$}
 \put(.6,2.5){$D$}
 \put(.4,1.9){$\Gamma$}
 \put(0.85,3.75){$\Omega _R$}
 \put(-0.03,0.675){$\Sigma _R$}
 \put(3.5,3.1){$u^i$}
\end{picture}
\vspace*{-1cm}
\end{center}
\caption{A cartoon showing the geometric features of the problem.  The
  bounded scatterer $D$ is covered by a thin coating giving rise to a
  GIBC on $\Gamma$.  An incident wave $u^i$ on this scatterer causes a scattered field $u$
  in the exterior of $D$.  The artificial boundary $\Sigma_R$ is introduced to truncate the domain resulting in a bounded computational domain $\Omega_R$ and is taken to be a circle for simplicity.}
\label{cartoon}
\end{figure}

In order to prove the well-posedness of this variational formulation, we now summarize some of the properties of the NtD map $N_R$ and GIBC boundary  map $G_\Gamma$. For a given function $f\in H^{-1/2}(\Sigma_R)$ the NtD map
$N_R:H^{-1/2}(\Sigma_R)\to H^{1/2}(\Sigma_R)$ is given by
\begin{equation}
(N_Rf)(\theta)=\sum_{n=-\infty}^{\infty}\gamma_n f_n\exp(in\theta) \, ,
  \label{NRdef}
  \end{equation}
where $f_n =\frac{1}{2\pi R} \int_{\Sigma_R}f(R,\theta)\exp (-in\theta) \;d\theta $ are the Fourier coefficients of $f$ on $\Sigma_R$ and
\[
\gamma_n=\frac{1}{k}\frac{H_n^{(1)}(kR)}{(H^{(1)}_n)'(kR)}
\, .
\]
According to \cite[page 97]{ColtonCakoni} there are constants $C_1>0$ and $C_2<\infty$ such that
\[
\frac{C_1}{\sqrt{1+n^2}}\leq |\gamma_n|\leq\frac{C_2}{\sqrt{1+n^2}} \, ,
\]
for all $n\in\mathbb{Z}$. Now define $\tilde{N}_R:H^{-1/2}(\Sigma_R)\to H^{1/2}(\Sigma_R)$ by
\[
(\tilde{N}_Rf)(\theta)=-\sum_{n=-\infty}^{\infty}\frac{R}{\sqrt{1+n^2}}\, f_n\, \exp(in\theta) \, .
\] 
Clearly $\tilde{N}_R$ is negative definite and
\[
-\int_{\Sigma_R}(\tilde{N}_Rf)\, \overline{f} \,d\theta=\sum_{n=-\infty}^{\infty}\frac{2\pi R^2}{\sqrt{1+n^2}}\, |f_n|^2=2\pi R\,  \Vert f\Vert_{H^{-1/2}(\Sigma_R)}^2\, .
\]
Also from \cite[page 97]{ColtonCakoni} we can obtain the asymptotic estimate
\[
\gamma_n=\frac{R}{n}\left(1+O(\frac{1}{n}) \right) \quad\mbox{ when  } n \to\infty\, ,
\]
so that $\gamma_n-{R}/{\sqrt{1+n^2}}=O(1/n^2)$ when  $ n \to\infty$. Hence, 
\begin{equation}\label{Nrsplit}
N_R=\hat{N_R}+\tilde{N}_R\, ,
\end{equation}
where 
 $\hat{N}_R:H^{-1/2}(\Sigma_R)\to H^{3/2}(\Sigma_R)$ is well-defined and bounded, in particular $\hat{N}_R:H^{-1/2}(\Sigma_R)\to H^{1/2}(\Sigma_R)$ is compact.

We next state some properties of the NtD map which follow from the properties of the better known DtN map.
\begin{lemma}\label{lemmaNtDsign} For all $f\in H^{-1/2}(\Sigma_R)$, it holds
$$\Re \left( \int_{\Sigma_R} N_R f\;\overline{f}\,dS\right) <0  \quad\mbox{and}\quad \Im\left( \int_{\Sigma_R} N_R f\;\overline{f}\,dS\right) \leq 0 \, .$$
\end{lemma}
\begin{proof}
The first inequality follows from \cite[Lemma 3.2]{koyama}, whereas the second is proved as follows: For any $f\in H^{-1/2}(\Sigma_R)$, we may write
\begin{eqnarray*}
\int_{\Sigma_R}N_R f\;\overline{f}\,dS 
&=& 2\pi R\sum_{n=-\infty}^\infty |f_n|^2 \frac{H_n^{(1)}(kR)\overline{H_n^{(1)'}(kR)}}{k\, |H_n^{(1)'}(kR)|^2} \, ,\\
\end{eqnarray*}
where, as above,  $$f_n = \frac{1}{2\pi R}\int_{\Sigma_R}f(R,\theta)\exp(-in\theta)\;d\theta$$ are the Fourier coefficients of $f$ on $\Sigma_R$.
Since $H_n^{(1)}(kR) = J_n(kR)+iY_n(kR)$, taking the imaginary part
\begin{eqnarray*}
 \Im \left( \int_{\Sigma_R} N_R f\;\overline{f}\,dS \right) &=& 2\pi R\sum_{n=-\infty}^\infty |f_n|^2 \, \frac{J'_n(kR) \, Y_n(kR)-J_n(kR)\, Y'_n(kR)}{k\, |H_n^{(1)'}(kR)|^2}\\
 &=& -4\sum_{n=-\infty}^\infty\frac{|f_n|^2}{k^2\, |H_n^{(1)'}(kR)|^2}\, ,
\end{eqnarray*}
by the Wronskian formula for Bessel functions (see e.g. \cite[9.1.16]{Stegun}).
\end{proof}
We note that the foregoing theory provides a direct proof that $N_R$ is an isomorphism as
a consequence of the Fredholm alternative thanks to Lemma \ref{lemmaNtDsign} and the splitting (\ref{Nrsplit}).
\begin{corollary}\label{corNtDisomorph}
The operator $N_R: H^{-1/2} (\Sigma _R) \to H^{1/2}(\Sigma _R)$ is an isomorphism.
\end{corollary}

Next we show that $G_\Gamma$ is well defined.
\begin{lemma}\label{lemma-Gok}
Under Assumption \ref{A} and the conditions (\ref{hyp-E!forwardprob}), the operator $G_{\Gamma  }\! : H^{-1}(\Gamma )\to  H^{1}(\Gamma )$ defined
in (\ref{Gdef}) is an isomorphism. In particular, $G_{\Gamma  }\! : H^{-1}(\Gamma )\to H^1(\Gamma )$ is well-defined, linear and continuous.
\end{lemma}
\begin{proof}
We start the proof by defining the bounded sesquilinear forms $a_{\Gamma }, b_{\Gamma  } :  H^1(\Gamma )\times  H^1(\Gamma ) \to\mathbb{C}$ by
 \[
 a _{\Gamma  } (\xi _1,\xi _2) = \int_{\Gamma } \beta  \,(\nabla_{\Gamma  }  \xi _1 \cdot  \nabla_{\Gamma  } \overline{\xi }_2+\xi _1 \, \overline{\xi }_2)\,dS 
 \quad\mbox{ and }\quad
 b_{\Gamma  } (\xi _1,\xi _2) = -\int_{\Gamma } (\lambda+\beta) \, \xi _1 \, \overline{\xi}_2 \,dS 
\, ,
\]
for any $ \xi_1,\xi_2\in H^1(\Gamma )$.  
Thanks to the Riesz representation theorem, we can consider the associated operators $A_{\Gamma  }, B_{\Gamma  }\!: H^1(\Gamma )\to  H^{-1}(\Gamma )$ that satisfy 
 \[
 (A_{\Gamma  } \xi _1 , \xi _2 ) _{ H^{-1} (\Gamma  )\times H^1 (\Gamma  ) } = a _{\Gamma  } (\xi _1,\xi _2) 
 \quad\mbox{ and }\quad
 (B_{\Gamma  } \xi _1 , \xi _2 ) _{ H^{-1} (\Gamma  )\times H^1  (\Gamma  ) } =  b_{\Gamma  } (\xi _1,\xi _2) 
 \, ,
\]
for any $ \xi_1,\xi_2\in H^1(\Gamma )$. 
Notice that, under assumption (\ref{hyp-E!forwardprob}),
 \[
 \Re \left( a _{\Gamma  } (\xi,\xi)\right)  = \int_{\Gamma } \Re (\beta )  \, (|\nabla_{\Gamma  }  \xi |^2+|  \xi |^2) \,dS \geq c \, \Vert \xi\Vert ^2 _{H^1(\Gamma) } \qquad \forall \xi \in H^1(\Gamma ) \, ,
 \] 
and, in consequence, using the Lax-Milgram theorem guarantees that $A_{\Gamma  }\! : H^1(\Gamma )\to  H^{-1}(\Gamma )$ is an isomorphism. 

Also notice that, by Rellich's theorem we know that $H^1 (\Gamma  )$ is compactly embedded into $L^2 (\Gamma  )$, so that $B_{\Gamma  }\! : H^1(\Gamma )\to  H^{-1}(\Gamma )$ is compact.

Moreover, under Assumption \ref{A},  $A_{\Gamma  }+B_{\Gamma  }\! : H^1(\Gamma )\to  H^{-1}(\Gamma )$
is injective. Therefore, by the Fredholm alternative, $A_{\Gamma  }+B_{\Gamma  }\! : H^1(\Gamma )\to  H^{-1}(\Gamma )$ is an isomorphism. 
\end{proof}

The next lemma shows that the  impedance boundary condition does not cause a loss of uniqueness for the scattering problem.
\begin{lemma}\label{lemmaGIBCsign}
For any $\eta\in H^{-1}(\Gamma )$, it holds that 
$$\Im \left( ( G_{\Gamma}\eta , {\eta})_{H^{1}( \Gamma)\times H^{-1}(\Gamma )} \right) 
\geq 0\, .$$
\end{lemma}
\begin{proof}
Using the variational definition of $G_{\Gamma}:H^{-1}(\Gamma)\rightarrow H ^{1} (\Gamma )$ for $\eta\in H^{-1}(\Gamma)$, and choosing the test function $\xi=G_{\Gamma}\eta\in H^1(\Gamma)$, gives
 $$\int_{\Gamma}\left(\beta\,\left|\nabla_{\Gamma}G_{\Gamma} \eta\right|^2 - \lambda\,\left| G_{\Gamma}\eta\right|^2\right)\,dS = \int_{\Gamma}\eta\;\overline{G_{\Gamma}\eta}\,dS.$$
This implies
 \begin{eqnarray*}
\Im\left( \int_{\Gamma}G_{\Gamma}\eta \;\overline{\eta }\,dS\right) &=& -\Im \left(\int_{\Gamma}\eta \;\overline{G_{\Gamma}\eta}\,dS\right) \\
&=& \int_{\Gamma}\left(-\Im ( \beta) \left|\nabla_{\Gamma}G_{\Gamma} \eta\right|^2 + \Im ( \lambda)\left|G_{\Gamma}\eta \right|^2\right)\,dS \, \geq\, 0 \, .
\end{eqnarray*}
The last inequality follows from the assumptions $\Im (\beta) \leq 0$ and $\Im(\lambda) \geq 0$ in (\ref{hyp-E!forwardprob}).
\end{proof}

Starting our analysis of (\ref{prob-FwdProbV}) we show that any solution is unique.

\begin{lemma}\label{lemma-FwdProbVok}
 Problem (\ref{prob-FwdProbV}) has at most one solution.
\end{lemma}
\begin{proof} Let us consider any solution of its homogeneous counterpart, that is,  $\bfv\in H({\rm{}div};\Omega_R)$ such that
\begin{equation}\label{prob-FwdProbV_0}
\int_{\Omega_R}\! (\frac{1}{k^2}\,\nabla\cdot\bfv\,\overline{\nabla\cdot\bfw}-\bfv\cdot\overline{\bfw})\,d\bfx 
+\int_{ \Sigma_R}\!\! N_R(\bfv\cdot\bfnu)\,\overline{\bfw\cdot\bfnu}\,dS
- \int_{\Gamma } G_{\Gamma }(\bfv\cdot\bfnu)\,\overline{\bfw\cdot\bfnu}\,dS=
 0
\, ,
\end{equation}
for all $\bfw\in H({\rm{}div};\Omega_R)$.
Since $\Omega_R$ is connected, by the Helmholtz decomposition theorem (see \cite[Th.2.7-Ch.I]{GiraultRaviart}) we can rewrite $\bfv$ as 
$$
\bfv =\nabla u +\boldsymbol{\psi} \qquad\text{in } \Omega _R \, ,
$$
for some $u\in H^1(\Omega _R)$ and $\boldsymbol{\psi}\in H_0 (\mathrm{div}^0;\Omega _R)$, where 
$$
H_0 (\mathrm{div}^0;\Omega _R)=\Big\{ \bfw\in H (\mathrm{div};\Omega _R)\, ;\,\, \nabla\cdot\bfw=0\,\,\text{in }\Omega _R ,\, \bfw\cdot\bfnu=0\,\,\text{on } \partial\Omega _R =\Gamma\cup \Sigma _R \Big\}\, .
$$
Then, the homogeneous problem (\ref{prob-FwdProbV_0}) may be rewritten as
\begin{equation}\label{prob-FwdProbV_0_bis}
\int_{\Omega_R}\! \Big(\frac{1}{k^2}\,\Delta u\,\overline{\nabla\cdot\bfw}-(\nabla u +\boldsymbol{\psi})\cdot\overline{\bfw}\Big) d\bfx
+\int_{ \Sigma_R}\!\! N_R\Big(\frac{\partial u }{\partial\bfnu}\Big)\,\overline{\bfw\!\cdot\!\bfnu}\,dS
- \int_{\Gamma } G_{\Gamma \! }\Big(\frac{\partial u }{\partial\bfnu}\Big)\,\overline{\bfw\cdot\bfnu}\,dS=
 0
\, ,
\end{equation}
for all $\bfw\in H({\rm{}div};\Omega_R)$. In particular, taking $\bfw\in\mathcal{C}^{\infty}_0 (\Omega _R)^2$ we deduce that
\begin{equation*}
 \displaystyle \nabla (\frac{1}{k^2}\,\Delta u + u) + \boldsymbol{\psi} = \boldsymbol{0} \qquad \text{in }\Omega _R \, .
\end{equation*}
Noticing that $\displaystyle \nabla (\frac{1}{k^2}\,\Delta u + u) = -\boldsymbol{\psi} \in H_0 (\mathrm{div}^0;\Omega _R)$ leads to 
$$
\displaystyle \Delta (\frac{1}{k^2}\,\Delta u + u) = 0\,\mbox{ in } \Omega _R \qquad \mbox{ and }\qquad \displaystyle\frac{\partial \, }{\partial\bfnu} (\frac{1}{k^2}\,\Delta u + u)=0 \,\mbox{ on } \partial\Omega_R=\Sigma _R\cup\Gamma \, .
$$ 
Hence, by uniqueness of the solution (up to a constant) of the
interior Neumann problem for Laplace operator in $\Omega_R$, we have
that $\displaystyle\frac{1}{k^2} \Delta u + u=C_u $ in $\Omega _R$ for
some constant $C_u\in \mathbb{C}$; in particular,
$\boldsymbol{\psi} = -\displaystyle \nabla (\frac{1}{k^2}\,\Delta u +
u) =\boldsymbol{0}$ in $\Omega _R$. Furthermore, $\tilde{u}=u-C_u\in
H^1(\Omega_R)$ satisfies
\begin{equation*}
 \displaystyle \Delta \tilde{u} + k^2\,\tilde{u} =  0\qquad \text{in }\Omega _R \, ,
\end{equation*}
so that $\frac{1}{k^2}\Delta u=-\tilde{u}$ in $\Omega_R$ and we deduce from 
(\ref{prob-FwdProbV_0_bis}) that
\begin{equation*}
 \begin{array}{cl}
 \displaystyle -\tilde{u} +  N_R\Big(\frac{\partial \tilde{u} }{\partial\bfnu}\Big) = 0 \quad & \text{on } \Sigma_R \, ,\\[1.5ex]
 \displaystyle  \tilde{u}+ G_{\Gamma  \! }\Big(\frac{\partial \tilde{u} }{\partial\bfnu}\Big) = 0 \quad & \text{on }\Gamma   \, .
 \end{array}
\end{equation*}
In consequence, by the invertibility of $N_R$ and $G_{\Gamma}$ and the uniqueness of solution of the forward problem with GIBC (see \cite[Th.2.1]{BCHprevious}), we have that $\tilde{u}=0$ in $\Omega _R$; that is to say,  ${u}=C_u$ in $\Omega _R$. Summing up, we conclude that
$$
\bfv =\nabla u + \boldsymbol{\psi} = \boldsymbol{0} \qquad\text{in } \Omega _R \, .
$$
\end{proof}

Using this uniqueness result and a suitable stable splitting of $\Hdiv$, we will be able to apply \cite[Theorem 1.2]{Buffa2005} to prove the well-posedness of the continuous problem. In particular, we write
\[
\Hdiv =H({\rm{}div}^0;\Omega_R )\oplus \nabla S \, ,
\]
where 
$$
H (\mathrm{div}^0;\Omega _R)=\Big\{ \bfw\in H (\mathrm{div};\Omega _R)\, ;\,\, \nabla\cdot\bfw=0\,\,\text{in }\Omega _R  \Big\}$$
and 
$$S=\Big\{p\in H_0^1(\Omega_R)\, ; \,\, \Delta u\in L^2(\Omega_R)\Big\} \, ,
$$
and $S$ is endowed with the inner product 
\[
(p,q)_S \, =\, \int_{\Omega} (\Delta p\, \Delta \overline{q}+\nabla p\cdot\nabla\overline{q}) \,d\mathbf{x} 
\, .
\]
Notice that the orthogonality of the above splitting implies that $\mathbf{u}\in H({\rm{}div}^0;\Omega)$ if, and only if, $\mathbf{u}\in \Hdiv$ and $(\mathbf{u},\nabla q)=0$ for all $q\in S$.
We also need to define the duality pairing $\llangle\cdot ,\cdot\rrangle$ between $\Hdiv$ and its dual space  $\Hdiv '$, with respect to the pivot space $L^2(\Omega_R)^2$, so that (note: this is defined without conjugation):
\[
\llangle \mathbf{u},\bfw\rrangle 
=\int_{\Omega_R} \mathbf{u}\cdot\bfw\,d\mathbf{x} 
\quad \forall \mathbf{u}\in \Hdiv' , \, \bfw\in \Hdiv \, .
\]
According to the above splitting, any $\mathbf{u}\in \Hdiv$ has the form $\mathbf{u}=\mathbf{u}_0+\nabla p$ for some $\mathbf{u}_0\in H({\rm{}div}^0;\Omega_R)$ and $p\in S$.
By the orthogonality of the splitting, and the fact that $\nabla\cdot\mathbf{u}_0=0$, we have  that
\[
\Vert \nabla p\Vert^2_{\Hdiv}+ \Vert \mathbf{u}_0\Vert^2_{L^2(\Omega_R )^2} = \Vert \mathbf{u}\Vert^2_{\Hdiv}\]
and, in particular, the splitting is stable.
Moreover, it allows us to define the linear continuous operator $\theta:\Hdiv\to \Hdiv$ by $\theta\mathbf{u}=\nabla p-\mathbf{u}_0$.

Next we define $A:\Hdiv\to \Hdiv'$ such that if $\mathbf{u}\in \Hdiv$ then $A\mathbf{u}\in \Hdiv'$ is given via the Riesz representation theorem by
\[
\llangle A\mathbf{u},\overline{\bfw}\rrangle=a(\mathbf{u},\bfw)\quad \mbox{ for all }\bfw \in \Hdiv\, ;
\]
recalling the definition of $a(\cdot,\cdot)$ in (\ref{adef}).

We can now state and prove the following result.
\begin{theorem}\label{th-FwdProbVok}
Problem (\ref{prob-FwdProbV}) is well-posed and the Babu\v{s}ka-Brezzi condition is satisfied.
\end{theorem}

\begin{proof}
Let $\mathbf{u}\in \Hdiv$ be split into $\mathbf{u}=\mathbf{u}_0+\nabla p$ for some $\mathbf{u}_0\in H({\rm{}div}^0;\Omega_R )$ and $p\in S$, and similarly $\bfw=\bfw_0+\nabla q \in \Hdiv$. Then 
\begin{equation}\label{aexpand}
\begin{array}{rcl}
a(\mathbf{u},\theta\bfw)&=&\displaystyle \int_{\Omega_R}\left(\frac{1}{k^2}
\Delta p\cdot\Delta \overline{q}+\nabla p\cdot\nabla \overline{q}+\mathbf{u}_0\cdot\overline{\bfw_0}\right)\,d\bfx\\
&&\quad +\displaystyle\int_{\Sigma_R}N_R((\mathbf{u}_0+\nabla p)\cdot\bfnu) (\overline{\nabla q-\bfw_0})\cdot\bfnu\,dS\\&&
- \int_{\Gamma}G_\Gamma ((\mathbf{u}_0+\nabla p)\cdot\bfnu)  (\overline{\nabla q-\bfw_0})\cdot\bfnu\,dS \\
&& \quad -2\,\displaystyle \int_{\Omega_R}\nabla p\cdot\nabla \overline{q}\,d\bfx \, .
\end{array}
\end{equation}
We expand the troublesome term
\begin{eqnarray*}
  \lefteqn{\int_{\Sigma_R} N_R ((\mathbf{u}_0+\nabla p)\cdot\bfnu)\,(\overline{\nabla q-\bfw_0})\cdot\bfnu\,dS}\\
  &=&
  -\int_{\Sigma_R} N_R (\mathbf{u}_0\cdot\bfnu)\,\overline{\bfw_0}\cdot\bfnu\,dS +
  \int_{\Sigma_R} N_R(\mathbf{u}_0\cdot\bfnu)\,\nabla \overline{q}\cdot\bfnu\,dS \\
&&\qquad-
 \int_{\Sigma_R} N_R(\nabla p\cdot\bfnu)\,\overline{\bfw_0}\cdot\bfnu\,dS+
 \int_{\Sigma_R} N_R(\nabla p\cdot\bfnu)\,\nabla \overline{q}\cdot\bfnu\,dS \, .
\end{eqnarray*}
So we can define the sesquilinear form
\[
a_+(\mathbf{u},\theta\bfw)=
\frac{1}{k^2} \, \int_{\Omega_R}\left(\Delta p\,\Delta \overline{q} + \nabla p\cdot\nabla \overline{q}  + \mathbf{u}_0\cdot\overline{\bfw_0}\right)\,d\bfx
- \int_{\Sigma_R}
N _R (\mathbf{u}_0\cdot\bfnu)\,\overline{\bfw_0}\cdot\bfnu\,dS  \, ,
\]
and use the remaining terms in (\ref{aexpand}) to define the sesquilinear form
\begin{eqnarray*}
b(\mathbf{u},\theta\bfw)&=&
   \int_{\Sigma_R}N_R(\mathbf{u}_0\cdot\bfnu)\,\nabla \overline{q}\cdot\bfnu\,dS 
-
   \int_{\Sigma_R} N_R(\nabla p\cdot\bfnu)\,\overline{\bfw_0}\cdot\bfnu\,dS
+
   \int_{\Sigma_R} N_R(\nabla p\cdot\bfnu)\,\nabla \overline{q}\cdot\bfnu\,dS 
\\&&\quad -  \int_{\Sigma_R} G_\Gamma(\mathbf{u}_0+\nabla p)\cdot\bfnu\,(\overline{\nabla q
- \bfw_0})\cdot\bfnu\,dS-2\, \int_{\Omega_R}\nabla p\cdot\nabla \overline{q}\,d\bfx \, .
\end{eqnarray*}
On the one hand, since $\Re( 
N _R )$ is negative definite (see Lemma \ref{lemmaNtDsign}) and the splitting of the space is stable, we have that there is a constant $\alpha>0$
independent of $\mathbf{u}\in H({\rm{}div};\Omega ) $ such that
\[
\Re(a_+(\mathbf{u},\theta\mathbf{u}))\, \geq\, \alpha\, \Vert\mathbf{u} \Vert_{\Hdiv}^2 \, .
\]
Now define the operator $T:\Hdiv\to \Hdiv'$ by
\[
\llangle T\mathbf{u},\overline{\bfw} \rrangle
=-b(\mathbf{u},\bfw ) \qquad \forall \mathbf{u},\bfw\in \Hdiv \, .
\]
Notice that $T$ is compact because each sesquilinear form in its definition is compact. For example, the sesquilinear form
\[
 \int_{\Sigma_R} N_R (\mathbf{u}_0\cdot\bfnu)\,\nabla \overline{q}\cdot\bfnu\,dS
\]
is compact by \cite[Theorem 1.3]{Kress77}, because the trace of  functions in $S$
defined into $H^{1/2}(\Sigma_R)$ is a compact operator; indeed, $S$ 
is a subset of $H^2(\Omega_R )$ due to our assumption of a smooth boundary $\partial \Omega_R = \Gamma \cap \Sigma _R$, and the normal derivative operator 
is compact from $H^2(\Omega_R  )$ into $H^{1/2}(\Sigma_R)$.  The remaining  sesquilinear forms are also compact by the same reasoning.
Hence $T$ is compact.
Then we conclude that
\[
\llangle (A+T)\mathbf{u},\theta\overline{\mathbf{u}}\rrangle\, =\,
a_+(\mathbf{u},\theta\mathbf{u})\geq \alpha \Vert \mathbf{u}\Vert_{\Hdiv}^2\, .
\]
Hence all the conditions of \cite[Assumption 1]{Buffa2005} are satisfied
and the existence of a unique solution to (\ref{prob-FwdProbV}) is shown by
\cite[Theorem 2.1]{Buffa2005}.  In addition this theorem shows that there is
an isomorphism $\tilde{\theta}:\Hdiv\to \Hdiv$
such that
\[
\llangle A\mathbf{u},\tilde{\theta}\overline{\mathbf{u}}\rrangle
\, \geq\,
 \alpha\, \Vert \mathbf{u}\Vert_{\Hdiv}^2 \, .
\]
This in turn implies that the Babu\v{s}ka-Brezzi condition is satisfied.
\end{proof}


\section{A Semidiscrete Problem}\label{semi}
In this section we consider a semidiscrete problem in which the GIBC
boundary operator is discretized but the space where we search for the
solution in $\Omega_R$ is not.  As discussed in the introduction,
we shall not consider the truncation of the NtD map here.

We shall need an additional assumption on the boundary operator
$G_\Gamma$.  In particular we need to know that it smooths the
solution on the boundary $\Gamma$, so we make the following second
assumption.
\begin{assumption}\label{G} 
For each $-1 \leq s \leq -1/2$, it holds that $G_\Gamma ( H^{s}(\Gamma)) \subseteq H^{s+2}(\Gamma)$ and  
there exists $C_s^{\Gamma}>0$ such that $\Vert G_\Gamma\lambda\Vert_{H^{s+2}(\Gamma)}\leq C_s^{\Gamma} \, \Vert \lambda\Vert_{H^{s}(\Gamma)}$ for any $\lambda\in H^s(\Gamma )$. 
\end{assumption}
\begin{remark}
Note that if Assumption \ref{G} holds, since $G^*_\Gamma\lambda=\overline{G_{\Gamma}\lambda}$, it also holds for $G_\Gamma^*$.
\end{remark}
Notice that Assumption \ref{G} further constrains the choice of the
coefficients $\beta$ and $\lambda$ in the generalized impedance
boundary condition on $\Gamma$.  Its role will be clarified in Lemma
\ref{lemma-GH}, where we apply Schatz's analysis \cite{Schatz} in
order to show that the finite element approximation of $G_\Gamma$,
defined shortly, converges.

%

On the inner boundary $\Gamma $ we consider a finite dimensional
subspace $S_H\subset H^1(\Gamma )$ of continuous piecewise polynomials
of degree at least $P$(with $P\geq 1$) on a mesh ${\cal
  T}_H^\Gamma$. We assume that the mesh ${\cal T}_H^\Gamma$ consists
of segments of the boundary $\Gamma$ of maximum length $H>0$, and that
it is regular and quasi-uniform: the latter means
that there exists a constant $\sigma^\Gamma \in [1,\infty)$ such that
\[
\frac{H}{H_e}\leq \sigma^\Gamma\quad \mbox{ for all edges }e\in\mathcal{T}^{\Gamma}_H \mbox{ and all }H>0 \, ,
\]
where $H_e$ denotes the arc length of the edge $e$ in the mesh.

\begin{remark}
Other choices of the discretization space on $\Gamma$  are possible.  For example we could use a trigonometric basis or a smoother spline space on $\Gamma$; these particular choices have advantages in that they would provide faster convergence of the UWVF scheme.  We shall not discuss them explicitly here but
will give an example of the use of a trigonometric space in Section \ref{num}.
\end{remark}

Then we approximate $G_{\Gamma  }\! :H^{-1}(\Gamma  )\to H^{1}(\Gamma  )$ by $G_{\Gamma  }^H\! :H^{-1}(\Gamma  )\to S_H$ using a discrete counterpart of (\ref{Gdef}). Indeed, each $\eta\in H^{-1}(\Gamma )$ is mapped onto $G^H_{\Gamma  }\eta\in S_{H}$,  the unique solution of
\begin{equation}
\int_{\Gamma }\! \Big(\beta \,\nabla_{\Gamma  } (G^H_{\Gamma  } \!\eta) \cdot \nabla_{\Gamma  }\overline{\xi} -\lambda\, G^H_{\Gamma  }\! \eta \,\overline{\xi}\Big) \, dS=\int_{\Gamma } \eta\,\overline{\xi}\,dS \qquad\forall\xi\in S_{H} .
\label{defGH}
\end{equation}
Notice that, as happens at continuous level, this definition can be applied for functions in a bigger space, which is now  $S_{H}'$  the dual space of $S_{H}$ with pivot space $L^2(\Gamma  )$. 
Indeed, Assumptions \ref{A} and \ref{G}, and the conditions on the coefficients in (\ref{hyp-E!forwardprob}), allow us to show that this operator is well-defined for $H$ small enough applying the usual Schatz's analysis~\cite{Schatz} of non-coercive sesquilinear forms. Such argument is quite standard and we do not give the details here: We just mention that it applies, not just because of the approximation properties of $S_H$, but since the operator $G_{\Gamma}: H^{-1}(\Gamma)\to H^1(\Gamma )$ can be understood as the solution operator for a bounded sequilinear form which is the superposition of a compact and a coercive sesquilinear forms; see the proof of Lemma \ref{lemma-Gok}.

\begin{lemma}\label{lemma-GH}
The operator $G^H_{\Gamma  }\! : S_{H}'\to S_{H}$ is an isomorphism  for any $H>0$ small enough. 
Furthermore, if $\lambda\in H^{-1/2}(\Gamma)$ is smooth enough that $G_\Gamma\lambda\in H^t(\Gamma)$ for some $t\in [1,P+1]$, then the following error estimate holds:
\[
\Vert (G_\Gamma-G_\Gamma^H)\lambda\Vert_{H^{s}(\Gamma)}\, \leq\,  CH^{t-s}\Vert G_\Gamma\lambda
\Vert_{H^{t}(\Gamma)}
\]
for any $s\in [0,1]$, and where $C$ is independent of $\lambda$.
\end{lemma}

We now consider the semidiscrete counterpart of problem (\ref{prob-FwdProbV}), which consists of computing $\bfv^{H}\in H({\rm div};\Omega_R)$ that satisfies
\begin{eqnarray}
&&\hspace*{-1cm}
\int_{\Omega_R}\!\! \Big(\frac{1}{k^2} \nabla\!\cdot\!\bfv^{H,M}\,\overline{\nabla\!\cdot\!\bfw}-\bfv^{H}\!\cdot\overline{\bfw} \Big) d\bfx
+\int_{ \Sigma _R} \!\!\! N_R(\bfv^{H}\!\cdot\!\bfnu)\,\overline{\bfw\!\cdot\!\bfnu}\,dS
-\int_{\Gamma } {G^H_{\Gamma\! }} (\bfv^{H}\!\cdot\!\bfnu) \, \overline{\bfw\!\cdot\!\bfnu}\,dS \, = \nonumber \\
&&\qquad\qquad
=\, \int_{\Gamma } { G^H_{\Gamma \! }  g} \,\overline{\bfw\!\cdot\!\bfnu}\,dS\qquad\mbox{ for all }\bfw\in H({\rm{}div};\Omega_R) \, .
\label{prob-FwdProbV_HM}
\end{eqnarray}
As at continuous level, it is useful to associate to the left hand side of (\ref{prob-FwdProbV_HM}) the sesquilinear form
 $a^H: H({\rm{}div};\Omega_R) \times H({\rm{}div};\Omega_R) \to \mathbb{C}$ defined by
\begin{equation}
  a^H(\mathbf{u},\mathbf{w})
=\displaystyle\int_{\Omega_R}\! (\frac{1}{k^2}\,\nabla\cdot\mathbf{u}\,\overline{\nabla\cdot\bfw}-\bfv\cdot\overline{\bfw})\,d\bfx 
  +\int_{\Sigma _R}\!\! N_R(\mathbf{u}\cdot\bfnu)\,\overline{\bfw\cdot\bfnu}\,dS- \int_{\Gamma }\!\! G^H_{\Gamma  \! }(\mathbf{u}\cdot\bfnu)\,\overline{\bfw\cdot\bfnu}\,dS\, , \label{aHdef}
\end{equation}
which is just the semidiscrete counterpart of $a(\cdot,\cdot )$.
%

Moreover,  to study the problem (\ref{prob-FwdProbV_HM}), we define the operator $A^H:\Hdiv\to\Hdiv'$ by
\[
\llangle A^H  \mathbf{u} , \overline{\bfw}\rrangle=a^H(\mathbf{u} ,\bfw)\quad \mbox{ for all } \mathbf{u} , \bfw\in \Hdiv \, .
\]
We can now show that $A^H$ converges to $A$ in norm.
\begin{lemma}\label{Ahconv}
For each $H>0$ sufficiently small, there is a constant $C$ such that 
$$\Vert A - A^H\Vert_{\Hdiv\to\Hdiv'}\,\leq\, C H
\, .$$
\end{lemma}
\begin{proof}
For any $\mathbf{u} ,\bfw\in \Hdiv$, from the own definitions of $A$ and $A^H$  we have that
  \begin{eqnarray*}
  \vert  \llangle (A-A^H) \mathbf{u} , \overline{\bfw} \rrangle \vert &=&\vert\int_\Gamma(G^H_{\Gamma}-G_{\Gamma})(\mathbf{u}\cdot\bfnu)\,\overline{\bfw\cdot\bfnu}\, dS \, \vert \\
    &\leq&\Vert (G^H_{\Gamma}-G_{\Gamma})(\mathbf{u}\cdot\bfnu)\Vert_{H^{1/2}(\Gamma)} \, \Vert\bfw\cdot\bfnu\Vert_{H^{-1/2}(\Gamma)}\\
    &\leq&C\Vert (G^H_{\Gamma}-G_{\Gamma})(\mathbf{u}\cdot\bfnu)\Vert_{H^{1/2}(\Gamma)} \, \Vert\bfw\Vert_{H({\rm{}div};\Omega_R)}\, .
  \end{eqnarray*}
  But, by Lemma~\ref{lemma-GH} and  Assumption \ref{G}, we conclude that
\begin{eqnarray*}
  \Vert (G_{\Gamma}^H-G_{\Gamma})(\mathbf{u} \cdot\bfnu)\Vert_{H^{1/2}(\Gamma)} &\leq&  CH\,  \Vert G_{\Gamma} ( \mathbf{u} \cdot\bfnu)\Vert_{H^{3/2} (\Gamma)}\\
& \leq&  CH\,  \Vert \mathbf{u} \cdot\bfnu \Vert_{H^{-1/2} (\Gamma)}
\,\leq\, CH\, \Vert \mathbf{u} \Vert_{\Hdiv} \, .
  \end{eqnarray*}
\end{proof}

Using \cite[Theorem 10.1]{Kress} we have the following result.
\begin{theorem}\label{lemma-E!sol:semidiscrGIBCfwdprob_vec} For all $H$ sufficiently small, the operator $A^H:\Hdiv\to\Hdiv'$
  is invertible and its inverse is bounded independently of $H$.
  Suppose $\mathbf{v}$ satisfies
  (\ref{prob-FwdProbV}) and $\mathbf{v}^H$ satisfies
  (\ref{prob-FwdProbV_HM}), then there is a constant $C$ independent of
  $H$ such that
  \begin{equation}\label{eq-boundsemidiscr}
  \Vert\mathbf{v}^H-\mathbf{v}\Vert_{\Hdiv}
\, \leq\,  
C \,\Big( \Vert (G^H_{\Gamma}-G_{\Gamma})(\mathbf{v}\cdot\bfnu)\Vert_{H^{1/2}(\Gamma)} 
+ \Vert (G^H_{\Gamma}-G_{\Gamma})g\Vert_{H^{1/2}(\Gamma)}\Big) \, .
  \end{equation}
\end{theorem}
\begin{proof}
  Recall that, as shown in the proof of Theorem \ref{th-FwdProbVok},  the operator $A: \Hdiv\to\Hdiv ' $ is an isomorphism. Further, Lemma \ref{Ahconv} shows the convergence of $A^H$ to $A$ in the norm $\Vert\cdot\Vert _{\Hdiv\to\Hdiv '}$. Then Theorem 10.1 of \cite{Kress} shows that, for $H$ small enough, $(A^H)^{-1}$ exists and is uniformly bounded in $H$. \\
Finally, to deduce the error bound (\ref{eq-boundsemidiscr}), we notice that
$$A^H (\bfv ^H-\bfv)\, = \, (A-A^H) \bfv + f^H - f
 \quad \mbox{  in } \Hdiv ' \, ,$$
where $A\bfv=f$ and $A^H\bfv^H=f^H$ in$\Hdiv '$:
\[
f^H(\mathbf{w})=\int_{\Gamma}G^H_\Gamma g\, \mathbf{w\cdot\bfnu}\,dS 
\quad\text{and}\quad 
f(\mathbf{w})=\int_{\Gamma}G_\Gamma g\, \mathbf{w}\cdot\bfnu \,dS \qquad\forall \mathbf{w}\in \Hdiv
\, .
\]
But we can estimate $\Vert (A-A^H)\bfv \Vert _{\Hdiv '} $ as in the proof of Lemma \ref{Ahconv}, which gives us the first term on the right hand side of (\ref{eq-boundsemidiscr}). Similarly, we can bound $\Vert f^H - f \Vert _{\Hdiv '}$, which gives us the second term on the right hand side of (\ref{eq-boundsemidiscr}).
  \end{proof}

Our final result of this section shows that $\bfv^H$ is smooth enough that the trace of $\nabla\cdot\bfv^H$ is well defined on line segments (edges of elements) in $\Omega _R$.

\begin{lemma}
  For each  $0\leq s<1/2$, there exists a constant $C$ (depending on $s$ but independent of $g\in H^{-1}(\Gamma)$) such that the solution $\mathbf{v}^H\in\Hdiv $ of (\ref{prob-FwdProbV_HM}) satisfies
$$
  \Vert\mathbf{v}^H\Vert_{H^{1/2+s}(\Omega_R)} \leq CH^{-s}\Vert g\Vert_{H^{-1}(\Gamma)} \quad\mbox{ and } \quad 
  \Vert \nabla\cdot\mathbf{v}^H\Vert_{H^{1}(\Omega_R)} \leq  C\Vert g\Vert_{H^{-1}(\Gamma)}\, .
$$
\end{lemma}
\begin{proof}
  Following the proof of the uniqueness result in Lemma \ref{lemma-FwdProbVok} and replacing there the operator $G_{\Gamma}$ by its discrete counterpart $G^H_{\Gamma}$,   we see that $\bfv^H=\nabla u^H$ where $u^H\in H^1 (\Omega_R)$ satisfies
  $$
\begin{array}{rcll}
    \Delta u^H+k^2u^H&=&0 &\mbox{ in }\Omega_R\, ,\\
\displaystyle -u^H+N_R(\frac{\partial u^H}{\partial \boldsymbol{\nu} })&=&0 &\mbox{ on }\Sigma_R\, ,\\[1ex]
\displaystyle u^H+G_\Gamma^H(\frac{\partial u^H}{\partial\boldsymbol{\nu} } )&=&G_\Gamma^H g\quad & \mbox{ on }\Gamma \, .
  \end{array}
$$
Since $S_H$ consists of continuous piecewise polynomials, 
we know that for each $0\leq s <1/2$ it holds $S_H \subseteq H^{1+s}(\Gamma) $ and, in particular,  $u^H|_{\Gamma}\in G_{\Gamma}^H (H^{-1} (\Gamma )) \subseteq S_H\subset H^{1+s}(\Gamma)$. Moreover, $N_R$ can be inverted to give the classical Dirichlet-to-Neumann map, so that $u^H$ can be extended to the exterior of $B_R$ as a radiating solution of Helmholtz equation in the whole $\Omega$. Identifying such extension with $u^H$ itself, we have that $u^H$ satisfies the exterior Dirichlet problem for Helmholtz equation in $\Omega$ and Dirichlet data $u^H|_{\Gamma} \in H^{1+s}(\Gamma)$. Hence, using a priori estimates for the exterior Dirichlet problem, $u^H\in H^{3/2+s}_{loc}(\Omega)$ 
and it satisfies 
$$
 \Vert u^H\Vert_{H^{3/2+s}(\Omega_R)}\, \leq \, \Vert u^H\Vert_{H^{3/2+s}_{loc}(\Omega )}\, \leq \, C\Vert u^H\Vert_{H^{1+s}(\Gamma)} \, = \,C\,\Vert G_\Gamma^H(\frac{\partial u^H}{\partial\boldsymbol{\nu}})-G_\Gamma^H g \Vert_{H^{1+s}(\Gamma)}\, .
$$
By our quasi-uniformity assumption on the mesh $\mathcal{T}_H^{\Gamma}$, we know 
that a standard inverse estimate holds for $S_H$ and hence
$$
\Vert u^H\Vert_{H^{3/2+s}(\Omega_R)}\, \leq \, CH^{-s}\,\Vert G_\Gamma^H(\frac{\partial u^H}{\partial\boldsymbol{\nu}}\! )-G_\Gamma^H g\Vert_{H^{1}(\Gamma)} \, \leq \, CH^{-s}\left(\Vert \frac{\partial u^H}{\partial\boldsymbol{\nu}}
\Vert_{H^{-1/2}(\Gamma)}+\Vert g\Vert_{H^{-1}(\Gamma)}\right).
$$
Now note that $\bfv ^H=\nabla u^H$ in $\Omega _R$, so that
$$
\Vert \bfv ^H\Vert_{H^{1/2+s}(\Omega _R)} \,\leq \, \Vert u^H\Vert_{H^{3/2+s}(\Omega_R)}\, \leq \, CH^{-s}\left(\Vert \bfv^H\!\cdot\!\bfnu\Vert_{H^{-1/2}(\Gamma)}+\Vert g\Vert_{H^{-1}(\Gamma)}\right) .
$$
Similarly, $\nabla\cdot\bfv^H=\Delta u^H=-k^2u^H$ and we deduce that 
  \begin{eqnarray*}
  \Vert \nabla\cdot\bfv^H\Vert_{H^1(\Omega_R)}&=&\, k^2\, \Vert u^H\Vert_{H^1(\Omega_R)}
  \leq \,C\,\Vert G_\Gamma^H(\frac{\partial u^H}{\partial\boldsymbol{\nu}}\! )-G_\Gamma^H g \Vert_{H^{1/2}(\Gamma)}
\\&\leq & C\left(\Vert \bfv^H\!\!\cdot\!\bfnu\Vert_{H^{-1/2}(\Gamma)}+\Vert g\Vert_{H^{-1}(\Gamma)}\right)\,.
  \end{eqnarray*}
We complete the estimate using the well-posedness of the semidiscrete problem and the continuity of normal traces from $\Hdiv$ into $H^{-1/2}(\Gamma )$.
\end{proof}

%

\section{A Trefftz DG Method}\label{trefftz}

We want to use a Trefftz discontinuous Galerkin method to approximate the semidiscrete problem (\ref{prob-FwdProbV_HM}).
In particular, in the scalar case, typical examples of Trefftz spaces for the Helmholtz problems are linear combinations of plane waves in different directions, or linear combinations of circular/spherical waves. The gradient of such solutions provides a basis for the vector problem. In the following we 
seek a Trefftz Discontinuous Galerkin (TDG) method to approximate the semidiscrete vector formulation of the problem (\ref{prob-FwdProbV_HM}). 

Let us introduce ${\cal T}_h$ a triangular mesh of $\Omega _R$,
possibly featuring hanging nodes, and allowing triangles to have
curvilinear edges if they share an edge with $\Gamma$ or $\Sigma_R$.
We write $h$ for the mesh width of ${\cal T}_h$, that is, $h = \max
_{K\in {\cal T}_h} h_K$ where $h_K$ the diameter of triangle $K$. On
${\cal T}_h$ we will define our TDG method. To this
  end,  we denote by ${\cal F}_h = \cup _{K\in {\cal T}_h } \partial
K$ the skeleton of the mesh ${\cal T}_h$, and set ${\cal F}_h^R =
{\cal F}_h \cap \Sigma_R$, ${\cal F}_h^{\Gamma} = {\cal F}_h \cap \Gamma $,
and ${\cal F}_h^I = {\cal F}_h \cap \Omega _R = {\cal F}_h \setminus (
{\cal F}_h^R \cup {\cal F}_h^{\Gamma}) $.  We also introduce some standard
DG notation: Write $\bfnu^+$, $\bfnu^-$ and $\bfnu_K$ for the exterior
unit normals on $\partial K ^+$, $\partial K ^-$ and $\partial K$,
respectively, where $K^+,K^-,K \in \mathcal{T}_h$. Let $u$ and
$\mathbf{v}$ denote a piecewise smooth scalar function and vector
field respectively on $\mathcal{T}_h$. On any edge $e\in
\mathcal{F}_h^I$ with $e=\partial K ^+\cap\partial K ^-$, where
$K^+,K^-\in \mathcal{T} _h$, we define
\begin{itemize}
\item the averages: $\avg{u} := \frac{1}{2} (u^-+u^+)$, $\avg{\mathbf{v} } := \frac{1}{2} (\bfv^-+\bfv^+)$; 
\item the jumps: $\jmp{ u }_{\bfnu}:=  u^-\bfnu^-+u^+\bfnu^+$, $\jmp{\mathbf{v} }_{\bfnu} :=       \mathbf{v }^-\cdot \bfnu^-+     \mathbf{v} ^+\cdot \bfnu^+$.
\end{itemize}
Furthermore, we will denote by $\nabla _{\! h}$ the elementwise application of $\nabla$, and by $\partial_{\bfnu,h}=\bfnu\cdot\nabla _{\! h} $ the element-wise application of $\partial_{\bfnu}$ on $\partial \Omega_R =\Gamma  \cup \Sigma_R$.

We next introduce a suitable Trefftz space to approximate the
semidiscrete problem written in vector form as
(\ref{prob-FwdProbV_HM}).
To this end, we introduce the vector TDG  spaces with local
number of plane wave directions $\{ p_K \}_{K \in\mathcal{T}_h}$, $p_K>3$, given by
$$
\mathbf{W}_h = \left\{ \bfw_{h} \in L^2( \Omega _R)^2 \, ; \,\, \bfw_h|_K \in \mathbf{W}_{p_K}(K)\,\,\forall K\in\mathcal{T} _h \right\} ,
$$
where each $\mathbf{W}_{p_K} ( K )$ \pmr{is the span of a set of $p_K$  linearly independent vector functions on} $K$ that enjoy the Trefftz property:
$$
\grad\ddiv \bfw_h + k^2 \bfw_h = 0 \quad\forall \bfw_h\in \mathbf{W}_{p_K} ( K ) \, .
$$ Then, for any $\mathbf{u}_h \in \mathbf{W}_h$ and an arbitrary
element $K\in\mathcal{T}_h$, we have the following integration by
parts formula:
\begin{eqnarray*}
0&=&\int_K(\grad\ddiv \mathbf{u}_h+k^2 \mathbf{u}_h)\cdot\overline{\bfw}_h\,d\bfx\\&=&
\int_K\! (-\ddiv \mathbf{u}_h\ddiv \overline{\bfw}_{h}+k^2\mathbf{u}_h\cdot\overline{\bfw}_h)\,d\bfx+\int_{\partial K}\!\!\ddiv \mathbf{u}_h\, \overline{\bfw}_h\cdot\bfnu_ K\,dS \, .
\end{eqnarray*}
Integrating by parts one more time
\[
\int_K( \mathbf{u}_h\cdot(\grad\ddiv \overline{\bfw}_h+k^2\overline{\bfw}_h)\,d\bfx
+\int_{\partial K}\!\! \ddiv \mathbf{u}_h\, \overline{\bfw}_h\cdot\bfnu_K\,dS-
\int_{\partial K}\!\!\mathbf{u}_h\cdot\bfnu_K\ddiv\overline{\bfw}_h\,dS \, =\, 0 \, .
\]
Now assuming that ${\bfw}_h\in \mathbf{W}_{p_K} (K)$,
we obtain the master equation that the fluxes are linked by
\[
\int_{\partial K} \!\! \ddiv \mathbf{u} _h \, \overline{\bfw}_h\!\cdot\!\bfnu_K\,dS-
\int_{\partial K}\!\! \mathbf{u} _h\! \cdot\!\bfnu_K\,\ddiv\overline{\bfw}_h\,dS\, =\, 0\, .
\]
This needs to be generalized to be applied to discontinuous trial and
test functions in $\mathbf{W}_h$.  Let $\widehat{\ddiv\mathbf{u}_h}$
and $\hat{\mathbf{u}}_h$ denote numerical fluxes computed from the
appropriate functions on either side of an edge $e$ in the mesh (or on
one side if the edge is on the boundary), as we will describe next.
We then write the extended master equation
\[
\int_{\partial K} \!\! \widehat{\ddiv \mathbf{u}_h} \, \overline{\bfw}_h\cdot\bfnu_K\,dS-
\int_{\partial K} \!\! \hat{\mathbf{u}}_h\cdot\bfnu_K \, \ddiv\overline{\bfw}_h\,dS\, =\, 0\, .
\]
Adding over all triangles in the mesh, $K\in \mathcal{T}_h$, we may write the sum using the sets ${\cal F}_h^R$, ${\cal F}_h^I$ and ${\cal F}_h^{\Gamma}$ as defined previously and obtain:
\begin{equation}
\begin{array}{l}
\displaystyle \int _{\mathcal{F} ^I_h}  
(\widehat{\ddiv \mathbf{u}}_h\, \jmp{\overline{\bfw}_h}_{\bfnu}-\hat{\mathbf{u}}_h\cdot\jmp{\ddiv\overline{\bfw}_h}_{\bfnu})\,dS
\\[1ex]
\hspace*{1cm} +\displaystyle \int _{\mathcal{F} ^R_h} 
 (\widehat{\ddiv \mathbf{u}}_h\, \overline{\bfw}_h\cdot\bfnu-\hat{\mathbf{u}}_h\cdot\bfnu \,\ddiv\overline{\bfw}_h) \,dS
\\[1ex]
\hspace*{2cm} -\displaystyle \int _{\mathcal{F} ^{\Gamma}_h}  
(\widehat{\ddiv \mathbf{u}}_h\,  \overline{\bfw}_h\cdot\bfnu-\hat{\mathbf{u}}_h\cdot\bfnu\,\ddiv\overline{\bfw}_h ) \,dS \, =\, 0 \, ,
\end{array}\label{dgsum}
\end{equation}
where the negative sign appears on the last term because of the use of an outward pointing normal on $\Gamma$.

Defining numerical fluxes
using conjugate variables, \pmr{we are led (see also \cite{buf07,git07})} to the following fluxes on edges in ${\cal F}_h^I$:
\begin{eqnarray*}
\widehat{\ddiv \mathbf{u}_h}=\avg{\ddiv\mathbf{u}_h}+ik\alpha_1 \, \jmp{\mathbf{u}_h}_{\bfnu}\, ,\\
\hat{\mathbf{u}}_h=\avg{\mathbf{u}_h}+\frac{\alpha_2}{ik}\,\jmp{\ddiv\mathbf{u}_h}_{\bfnu} \, .
\end{eqnarray*}
Here $\alpha_1$ and $\alpha_2$ \pmr{are strictly positive real numbers on  each edge $e\in \mathcal{F}_h^I$. 
For the Ultra Weak Variational Formulation that we usually use, $\alpha_1=\alpha_2=1/2$~\cite{cessenat03}.  More generally they could be mesh dependent~\cite{HMPhp,git07}.  Since our numerical results are for constant $\alpha_1$ and $\alpha_2$ we shall not investigate these more general cases further.}

For  the edges on the outer boundary, ${\cal F}_h^R$, following \cite{shelvean_phd} we take
\begin{eqnarray*}
\widehat{\ddiv \mathbf{u}_h}&=&-k^2N_R(\mathbf{u}_h\cdot\bfnu)+\delta{i}k\, N_R^{*}(\ddiv \mathbf{u}_h+k^2N_R(\mathbf{u}_h\!\cdot\!\bfnu)) \, ,\\
\hat{\mathbf{u}}_h&=&\mathbf{u}_h+\frac{\delta}{ik}\,\big(\ddiv \mathbf{u}_h+k^2N_R (\mathbf{u}_h\!\cdot\!\bfnu)\big) \bfnu \, .
\end{eqnarray*}
where $N_R^{*}$ is the $L^2(\Sigma_R)$-adjoint of $N_R$, and $\delta>0$ is a parameter to be chosen.  

Furthermore, for edges on the impedance boundary, ${\cal F}_h^D$, we consider
 \begin{eqnarray*}
\widehat{\ddiv \mathbf{u}}_h&=&-k^2G_\Gamma^H(\mathbf{u}_h\cdot\bfnu+g)-ik\,\tau\, G^{H,*}_{\Gamma} (\ddiv \mathbf{u}_h+k^2G_H(\mathbf{u}_h\!\cdot\!\bfnu +g))\, ,\\
\hat{\mathbf{u}}_h&=&\mathbf{u}_h-\frac{\tau}{ik}\, \big(\ddiv \mathbf{u}_h+k^2G_H(\mathbf{u}_h\cdot\bfnu+g)\big) \bfnu\, .
\end{eqnarray*}
where $G_{\Gamma}^{H,*}$ is the $L^2(\Gamma )$-adjoint of $G_{\Gamma}^H$, and $\tau>0$ is a parameter to be chosen. Note the sign change compared to the  fluxes on the outer boundary $\Sigma _R$ because of the outward pointing $\bfnu$.

Using these fluxes in (\ref{dgsum}) leads us to defining the sesquilinear form
\begin{eqnarray}
&& a^H_h (\bfu,\bfw)=\int_{{\cal F}_h^I}\left(\avg{\ddiv\mathbf{u}}\overline{\jmp{\bfw}}_{\bfnu}-\avg{\mathbf{u}}\cdot\overline{\jmp{\ddiv\bfw}}_{\bfnu}\right)\,dS\nonumber\\
&&\quad -\frac{1}{ik}\int_{{\cal F}_h^I}\alpha_2\jmp{\ddiv\mathbf{u}}_{\bfnu}\cdot\overline{\jmp{\ddiv\bfw}}_{\bfnu}\,dS-\int_{\Sigma_R}\left(k^2N_R(\mathbf{u}\cdot\bfnu)\overline{\bfw\cdot\bfnu}+\mathbf{u}\cdot\bfnu\overline{\ddiv\bfw}\right)\,dS\nonumber\\
&&\quad -\frac{1}{ik}\int_{\Sigma_R}\delta(\ddiv\mathbf{u}+k^2N_R(\mathbf{u}\cdot\bfnu))\overline{(\ddiv\bfw+k^2N_R(\bfw\cdot\bfnu))}\,dS\nonumber\\
&& +\int_{\Gamma}\left(k^2G_\Gamma^H(\mathbf{u}\cdot\bfnu)\,\overline{\bfw\cdot\bfnu}+\mathbf{u}\cdot\bfnu \, \overline{\ddiv\bfw}\right)\,dS+ik\int_{{\cal F}_h^I}\alpha_1\jmp{\mathbf{u}}_{\bfnu}\overline{\jmp{\bfw}}_{\bfnu}\,dS\nonumber\\
&&\quad -\frac{1}{ik}\int_{\Gamma}\tau(\ddiv\mathbf{u}+k^2G_\Gamma^H(\mathbf{u}\cdot\bfnu))\overline{(\ddiv\bfw+k^2G_\Gamma^H(\bfw\cdot\bfnu))}\,dS \, ,
\end{eqnarray}
and the antilinear functional
\[
f^H_h(\bfw)=-\frac{1}{ik}\int_\Gamma\tau k^2 G_\Gamma^H(g)\, (\ddiv\overline{\bfw} +k^2G_\Gamma^H(\overline{\bfw}\cdot\bfnu)) \,dS+
\int_\Gamma k^2G_\Gamma^H(g)\,\overline{\bfw}\cdot\bfnu\,dS
\]
Then the discrete problem we wish to solve is to find $\mathbf{v}_h^H\in W_h$ such that
\begin{eqnarray}\label{DiscreteProblem}
a^H_h (\mathbf{v}_h^H,\bfw)&=& f_h^H (\bfw )\mbox{ for all }\bfw\in \mathbf{W}_h.
\end{eqnarray}
We start by showing that this problem has a unique solution for any $h>0$ and $k>0$ and $H$ small enough.
It is  useful to define the sesquilinear forms
\begin{eqnarray*}
a_{0,h}^H(\mathbf{u},\bfw)&=&
\int _{\mathcal{F}_h^I} 
(\avg{\ddiv\mathbf{u}} \jmp{\overline{\bfw}}_{\bfnu}-\avg{\mathbf{u}}\cdot\jmp{\ddiv\overline{\bfw}}_{\bfnu}) \,dS
\\
&&-\int_{\Sigma_R}\mathbf{u}\!\cdot\!\bfnu\,\ddiv\overline{\bfw} \,dS +\int_{\Gamma}
\mathbf{u}\!\cdot\!\bfnu\,\ddiv\overline{\bfw}\,dS \, ,
\end{eqnarray*}
and
\begin{eqnarray*}
b^H_h(\mathbf{u} ,\bfw)  &\! = \! &
ik\int_{{\cal F}_h^I}\alpha_1\jmp{\mathbf{u}}_{\bfnu}\overline{\jmp{\bfw}}_{\bfnu}\,dS-\frac{1}{ik}\int_{{\cal F}_h^I}\alpha_2\jmp{\ddiv\mathbf{u} }_{\bfnu}\cdot\overline{\jmp{\ddiv\bfw}}_{\bfnu}\,dS\nonumber\\
&& -\frac{1}{ik}\int_{\Sigma_R}\!\!\! \delta\, (\ddiv\mathbf{u}+k^2N_R(\mathbf{u}\!\cdot\!\bfnu))\overline{(\ddiv\bfw+k^2N_R (\bfw\!\cdot\!\bfnu))}\,dS
-\int_{\Sigma_R}\!\!\! k^2 N_R(\mathbf{u}\!\cdot\!\bfnu)\, \overline{\bfw\! \cdot\! \bfnu}\,dS\nonumber\\&&
+\int_{\Gamma}\!\! \!k^2\, G_\Gamma^H(\mathbf{u}\!\cdot\!\bfnu)\, \overline{\bfw\!\cdot\!\bfnu}\,dS-\frac{1}{ik}\int_{\Gamma}\! \!\! \tau\, (\ddiv\mathbf{u}+k^2G_\Gamma^H(\mathbf{u}\!\cdot\!\bfnu))\overline{(\ddiv\bfw+k^2G_\Gamma^H(\bfw\!\cdot\!\bfnu))}\,dS
\end{eqnarray*}
Obviously $a^H_h(\mathbf{u} ,\bfw)=a_{0,h}^H (\mathbf{u} ,\bfw)+b_h^H(\mathbf{u} ,\bfw)$.

We start by rewriting $a_{0,h}^H$ in an equivalent form using the DG Magic Lemma \cite{EM}.  In particular since $\bfw$ satisfies the Trefftz condition, for all $\mathbf{u},\bfw\in \mathbf{W}_h$ we have
\begin{eqnarray*}
\lefteqn{\sum_{K\in{\cal T}_h}\int_K(\ddiv\mathbf{u}\, \overline{\ddiv\bfw}-k^2\mathbf{u}\cdot\overline{\bfw})\,d\bfx=\sum_{K\in {\cal T}_h}\int_{\partial K}\mathbf{u}\cdot\bfnu\overline{\ddiv \bfw}\,dS}\\&
=&\int_{\mathcal{F}_h^I}(\avg{\mathbf{u}}\cdot\overline{\jmp{\ddiv\bfw}}_{\bfnu}+\jmp{\mathbf{u}}_{\bfnu}\overline{\avg{\ddiv\bfw}})\,dS
+\int_{\Sigma_R}\!\!\!\mathbf{u}\cdot\bfnu\, \overline{\ddiv \bfw}\,dS
-\int_{\Gamma}\mathbf{u}\cdot\bfnu\, \overline{\ddiv \mathbf{w}}\,dS\, .
\end{eqnarray*}
Using this equality in the definition of $a_{0,h}^H(\mathbf{u},\bfw)$ we see that $a_{0,h}^H(\mathbf{u},\bfw)=a_{1,h}^H(\mathbf{u},\bfw)$ where
\[
a_{1,h}^H(\mathbf{u},\bfw):=-\sum_{K\in{\cal T}_h}\int_K(\ddiv\mathbf{u}\, \overline{\ddiv\bfw}-k^2\,\mathbf{u}\cdot\overline{\bfw})\,d\bfx+ \int_{\mathcal{F}^I_h}
(\avg{\ddiv\mathbf{u}} \jmp{\overline{\bfw}}_{\bfnu}+\jmp{\mathbf{u}}_{\bfnu}\cdot\overline{\avg{\ddiv\bfw}}) \,dS\, .
\]
Then choosing $\bfw=\mathbf{u}$ we immediately have
\[
\Im(a_{0,h}^H(\mathbf{u},\mathbf{u}))
=\Im(a_{1,h}^H(\mathbf{u},\mathbf{u}))=0 \, .
\]

Turning to the sesquilinear form $b_h^H(\mathbf{u},\bfw)$ if we choose  $\bfw=\mathbf{u}$ then
\begin{eqnarray*}
b_h^H (\mathbf{u},\mathbf{u}) & \!\! = \!\! &
ik\int_{{\cal F}_h^I}\alpha_1\, |\jmp{\mathbf{u}}_{\bfnu}|^2\,dS -\frac{1}{ik}\int_{{\cal F}_h^I}\alpha_2\, |\jmp{\ddiv\mathbf{u}}_{\bfnu}|^2\,dS \nonumber\\
&&\quad -\frac{1}{ik}\int_{\Sigma_R}\!\! \delta\, |\ddiv\mathbf{u}+k^2N_R (\mathbf{u}\cdot\bfnu)|^2\,dS
-\int_{\Sigma_R}\!\! k^2\, N_R(\mathbf{u}\cdot\bfnu)\overline{\mathbf{u}\cdot\bfnu}\,dS\nonumber\\&&
+\int_{\Gamma}  k^2\, G_\Gamma^H(\mathbf{u}\cdot\bfnu)\overline{\mathbf{u}\cdot\bfnu}\,dS
-\frac{1}{ik}\int_{\Sigma_R}\!\! \tau\, |\ddiv\mathbf{u}+k^2G_\Gamma^H(\mathbf{u}\cdot\bfnu))|^2\,dS \, .
\end{eqnarray*}

Thus since $\alpha_1$, $\alpha_2$, $\delta$ and $\tau$ are real valued 
\begin{eqnarray*}
\Im (b^H_h (\mathbf{u},\mathbf{u}) &\! =\! & \int _{\mathcal{F}_h^I} 
 (k\alpha_1\, |\jmp{\mathbf{u}}_{\bfnu}|^2+\frac{\alpha_2}{k}\, |\jmp{\ddiv\mathbf{u}}_{\bfnu}|^2) \,dS
\\
&&-\Im ( \int_{\Sigma_R} \!\! k^2 N_R(\mathbf{u}\!\cdot\! {\bfnu})\overline{\mathbf{u}\!\cdot\! {\bfnu}}\,dS ) + \Im (\int_{\Gamma} k^2G^H_{\Gamma}(\mathbf{u}\!\cdot\! \bfnu)\, \overline{\mathbf{u}\!\cdot\! \bfnu}\,dS) \\
&&+\int_{\Sigma_R}\frac{\delta}{k}\,\left|\ddiv {\mathbf{u}}+k^2 N_R (\mathbf{u}\!\cdot\! {\bfnu})\right|^2\,dS
+\int_{\Gamma}\frac{\tau}{k}\,\left|\ddiv \mathbf{u}+k^2 G^H_{\Gamma} (\mathbf{u}\!\cdot \! {\bfnu})\right|^2\,dS\, .
\end{eqnarray*}
Note that, by Lemmas~\ref{lemmaNtDsign} and \ref{lemmaGIBCsign} (which is stated for $G_{\Gamma}$ but a similar reasoning shows that it also holds for $G^H_{\Gamma}$),
\[
\Im (\int_{\Sigma_R} N_R (\mathbf{u}\!\cdot\! {\bfnu})\, \overline{\mathbf{u}\!\cdot\! {\bfnu}}\,dS) 
- \Im (\int_{\Gamma} G^H_{\Gamma}(\mathbf{u}\!\cdot\! {\bfnu})\, \overline{\mathbf{u}\!\cdot\! {\bfnu}}\,dS) \, \leq\, 0 \, ,
\]
and so
\begin{eqnarray*}
\Im(b^H_h(\mathbf{u},\mathbf{u}))& \!\!\geq\!\! & \int_{\mathcal{F}^I_h}
(k\alpha_1\, |\jmp{\mathbf{u}}_{\bfnu}|^2+\frac{\alpha_2}{k}\, |\jmp{\ddiv\mathbf{u}}_{\bfnu}|^2) \,dS
\\
&&+\int_{\Sigma_R}\frac{\delta}{k}\,\left|\ddiv {\mathbf{u}}+k^2 N_R (\mathbf{u}\!\cdot\! {\bfnu})\right|^2\,dS
+\int_{\Gamma}\frac{\tau}{k}\,\left|\ddiv \mathbf{u}+k^2 G^H_{\Gamma} (\mathbf{u}\!\cdot \! {\bfnu})\right|^2\,dS\, .
\end{eqnarray*}
We may thus define the mesh-dependent semi-norm $\|\bfw\|_{DG} = \sqrt{\Im(b^H_h (\bfw,\bfw))}$
 for any function $\bfw\in\mathbf{W}^s(\mathcal{T}_h)$ where $\mathbf{W}^s(\mathcal{T}_h)$ is defined as follows and contains $\mathbf{W}_h$:
$$
\mathbf{W}^s (\mathcal{T}_h) = \{ \bfw\in L^2 (\Omega _R)^2 ; \, \bfw|_K \in H^{1/2+s}(\mathrm{div};K) \, \text{s.t. } \nabla \nabla\cdot \bfw + k^2 \bfw = 0 \,\text{in } K , \, \forall K\in\mathcal{T}_h \, \} \, ,
$$
for any $s\in\mathbb{R}$ with $s>0$.   We now have the following result.
\begin{lemma}\label{lemDGnormbound}
For any $s>0$ and all $H>0$ small enough, the semi-norm $\Vert\cdot\Vert_{DG}$ is a norm on $\mathbf{W}^s (\mathcal{T}_h)$, and
  \begin{eqnarray}\label{DGnormbound}
\Vert \mathbf{w} \Vert_{DG}^2&\geq &
\int_{\mathcal{F}_h^I} 
( k\alpha_1\, |\jmp{\mathbf{w} }_{\bfnu} |^2
+\frac{\alpha_2}{k}\, |\jmp{\ddiv\mathbf{w} }_{\bfnu} |^2) \,dS
\\
&&
+\int_{\Sigma_R}\frac{\delta}{k}\,  \left|\ddiv \mathbf{w} +k^2 \, N_R(\mathbf{w} \!\cdot\! {\bfnu})\right|^2\,dS
+\int_{\Gamma }\frac{\tau}{k} \, \left| \ddiv \mathbf{w} +k^2\, G^H_{\Gamma} (\mathbf{w} \!\cdot\! {\bfnu})\right|^2\,dS\, .\nonumber
\end{eqnarray}
\end{lemma}
\begin{proof}
On one hand, if $\|\bfw\|_{DG}=0$ for some $\bfw\in \mathbf{W}^s
(\mathcal{T}_h)$, then $\nabla\nabla\cdot\bfw+k^2\bfw=0$ in $\Omega_R$
and $\nabla\cdot\bfw + k^2 N_R(\bfw\cdot\bfnu)=0$ on $\Sigma_R$,
$\nabla\cdot\bfw + k^2 G^H_{\Gamma}(\bfw\cdot\bfnu)=0$ on $\Gamma$.
The well-posedness of the semi-discrete problem for all $H$ small
enough, Theorem~\ref{lemma-E!sol:semidiscrGIBCfwdprob_vec}, implies that
$\bfw =\boldsymbol{0}$, so that the semi-norm $\|\cdot\|_{DG}$ is a
norm on $\mathbf{W}^s(\mathcal{T}_h)$. On the other hand, the norm
bound follows from the argument preceding the lemma.
\end{proof}

We now have the existence and uniqueness of solution for 
the discrete problem.
\begin{proposition}\label{Prop:DiscrExistUnique}
 For all $H$ small enough and any $h>0$ and $k>0$ there exists a unique solution $\mathbf{v}_h^H\in \mathbf{W}_h$ to the problem~(\ref{DiscreteProblem}) for every $g\in H^{-1}(\Gamma) $. 
 \end{proposition}
\begin{proof}
By the finite dimension of the space $\mathbf{W}_h$, it suffices to
show uniqueness of solution. To this end, we consider a solution of
the homogeneous problem, that is,
$\mathbf{v}_h^H\in\mathbf{W}_h$ such that
$a_h^H(\bfv^H_h,\bfw)=0$ for any $\bfw\in \mathbf{W}_h$. Then
$a_h^H(\bfv^H_h,\bfv_h^H)=0$, so that $\|\bfv^H_h\|_{DG} ^2=
\Im ( b^H_h (\bfv_h^H,\bfv_h^H) ) = \Im ( a(\bfv_h^H,\bfv_h^H)
) =0$. Hence $\bfv^H_h=0$ since $\|\cdot\|_{DG}$ is a norm on
$\mathbf{W}_h\subseteq W^s( \mathcal{T}_h)$ (see Lemma \ref{lemDGnormbound}).
\end{proof}

\subsection{A bound of the approximation error in the mesh-dependent norm $\Vert\cdot\Vert_{DG}$}
We now introduce the mesh-dependent norm $\|\cdot\|_{DG^+}$ on $\mathbf{W}^s(\mathcal{T}_h)$ as
\begin{eqnarray*}
 \|\bfw\|^2_{DG^+}&=&\|\bfw\|^2_{DG}+k^{-1}\int_{{\cal F}_h^I} \alpha^{-1}_1\, |\avg{\ddiv {\bfw}}|^2\,dS 
+ k\int_{{\cal F}_h^I} \alpha^{-1}_2\, |\avg{\bfw}|^2\,dS\\
 && + k\int_{\Sigma_R}\!\! \delta^{-1}\, |\bfw\cdot\bfnu|^2\,dS + k\int_{\Gamma} \tau^{-1}\, |\bfw\cdot\bfnu|^2\,dS.
\end{eqnarray*}

\begin{proposition}\label{BoundedForm}
For any $\mathbf{u}, \mathbf{w}\in \mathbf{W}^s(\mathcal{T}_h)$, we have 
$$a^H _h (\mathbf{u},\mathbf{w})\,\leq\,  2\, \|\mathbf{v}\|_{DG}\, \|\mathbf{w} \|_{DG^+}\, .$$ 
\end{proposition}
\begin{proof}
Using integration by parts, the Trefftz property of $\mathbf{u}\in \mathbf{W}^s(\mathcal{T}_h)$, and the DG Magic Lemma, we have 
\begin{eqnarray}\label{DGMagic1}
&& \sum_{K\in\mathcal{T}_h}\int_{K}\left(\nabla\cdot\mathbf{u}\, \overline{\nabla\cdot\bfw}-k^2\mathbf{u}\cdot\overline{\bfw}\right)\,d\bfx 
\, = \,  \sum_{K\in\mathcal{T}_h}\int_{\partial K}\nabla\cdot\mathbf{u}\,\overline{\bfw\cdot\bfnu}\,dS\nonumber\\
 && \quad = \int_{{\cal{F}}_h^I}\avg{\nabla\cdot\mathbf{u}}\overline{\jmp{\bfw}}_{\bfnu}\,dS + \int_{{\cal{F}}_h^I}\jmp{\ddiv{\mathbf{u}}}_{\bfnu}\cdot\overline{\avg{\bfw}}\,dS-\int_{\Gamma}\!\! \nabla\cdot\mathbf{u}\,\overline{\bfw\cdot\bfnu}\,dS \nonumber\\
 && \qquad 
+ \int_{\Sigma_R}\!\! \ddiv {\mathbf{u}}\,\overline{\bfw\cdot\bfnu}\,dS.
\end{eqnarray}
Substituting the expression for $ \sum_{K\in\mathcal{T}_h}\int_{K}\left(\nabla\cdot\mathbf{u}\, \overline{\nabla\cdot\bfw}-k^2\mathbf{u}\cdot\overline{\bfw}\right)\,d\bfx $
from equation~(\ref{DGMagic1}) above into the expression for $a_{1,h}^H (\mathbf{u},\bfw)$ leads to 
\begin{eqnarray*}
a_{1,h}^H(\mathbf{u},\bfw) &=& \int_{{\cal{F}}_h^I}\jmp{\mathbf{u}}_{\bfnu}\overline{\avg{\ddiv{\bfw}}}\,dS - \int_{{\cal{F}}_h^I}\jmp{\ddiv{\mathbf{u}}}_{\bfnu}\cdot\overline{\avg{\bfw}}\,dS +\int_{\Gamma} \ddiv{\mathbf{u}}\,\overline{\bfw\cdot\bfnu}\,dS 
- \int_{\Sigma_R}\!\! \ddiv{\mathbf{u}}\,\overline{\bfw\cdot\bfnu}\,dS.
\end{eqnarray*}
Then since $a^H_h (\mathbf{u},\bfw)=b^H_h(\mathbf{u},\bfw)+a_{1,h}^H(\mathbf{u},\bfw)$, we have that
 \begin{eqnarray*}
a^H_h (\mathbf{u},\bfw)& \! =\! & 
ik\int_{{\cal F}_h^I}\alpha_1\jmp{\mathbf{u}}_{\bfnu}\overline{\jmp{\bfw}}_{\bfnu}\,dS-\frac{1}{ik}\int_{{\cal F}_h^I}\alpha_2\jmp{\ddiv\mathbf{u}}_{\bfnu}\cdot\overline{\jmp{\ddiv\bfw}}_{\bfnu}\,dS\\
&& -\int_{\Sigma_R}\left(\ddiv{\mathbf{u}}+k^2N_R (\mathbf{u}\cdot\bfnu)\right)\overline{\bfw\cdot\bfnu}\,dS + \int_{\Gamma}\left(\ddiv{\mathbf{u}}+k^2G_\Gamma^H(\mathbf{u}\cdot\bfnu)\right)\overline{\bfw\cdot\bfnu}\,dS\nonumber\\
&& -\frac{1}{ik}\int_{\Sigma_R}\delta(\ddiv\mathbf{u}+k^2N_R (\mathbf{u}\cdot\bfnu))\overline{(\ddiv\bfw+k^2N_R (\bfw\cdot\bfnu))}\,dS\nonumber\\&&
{-} \frac{1}{ik}\int_{\Gamma}\tau(\ddiv\mathbf{u}+k^2G_\Gamma^H(\mathbf{u}\cdot\bfnu))\overline{(\ddiv\bfw+k^2G_\Gamma^H(\bfw\cdot\bfnu))}\,dS\nonumber\\
&&+\int_{{\cal{F}}_h^I}\jmp{\mathbf{u}}\overline{\avg{\ddiv{\bfw}}}\,dS - \int_{{\cal{F}}_h^I}\jmp{\ddiv{\mathbf{u}}}\cdot\overline{\avg{\bfw}}\,dS.
\end{eqnarray*}
By using the weighted Cauchy Schwarz inequality, we get the result.
\end{proof}

We now state a quasi-optimal error estimate with respect to the $DG$ and $DG^+$ norms.
\begin{theorem}\label{TheoremQuasiOptimal}
Assume $\mathbf{v}\in\mathbf{W}^s(\mathcal{T}_h)$ is the analytical solution to problem~(\ref{prob-FwdProbV_HM}), and $\mathbf{v}_h^H$ the unique solution to problem~(\ref{DiscreteProblem}).
Then $$\|\mathbf{v} - \mathbf{v}_h^H\|_{DG}\, \leq\, 2\inf_{\mathbf{w}_h\in\mathbf{W}_h}\, \|\mathbf{v} - \mathbf{w}_h\|_{DG^+}\, .$$
\end{theorem}
\begin{proof}
 Since $\bfv-\bfv_h^H\in \mathbf{W}^s(\mathcal{T}_h)$, for all $\bfw_h\in\mathbf{W}_h$ we have
 \begin{eqnarray*}
 \|\bfv-\bfv^H_h\|^2_{DG} &=& \Im ( a^H_h (\bfv-\bfv_h,\bfv-\bfv_h) ) 
\, \leq 
\, 2\|\bfv - \bfv_h\|_{DG}\, \|\bfv-\bfw_h\|_{DG^+} \, ,
  \end{eqnarray*}
  where the last inequality follows from the consistency of the discrete scheme, and continuity of the sesquilinear form in Proposition~\ref{BoundedForm}.
\end{proof}


\subsection{An error bound in a mesh-independent norm}

We next derive a bound of the approximation error in terms of a mesh-independent norm. Ideally this would be the $L^2(\Omega_R)^2$-norm, but as in the case of Maxwell's equations (see \cite{HMP_Maxwell}) this is not possible and we
derive an estimate in the $\Hdiv '$-norm. 
To this end, we start bounding some mesh-independent norm in terms of the mesh-dependent norm $\Vert\cdot\Vert_{DG}$ in the the vector space $\mathbf{W}^s (\mathcal{T}_h)$ for $s\in \mathbb{R}$, $s>0$, which contains the Trefftz vector space $\mathbf{W}_h$.

Recall that any function in this space may be written using the $L^2(\Omega_R)^2$-orthogonal Helmholtz decomposition
\begin{equation}
  \bfw = \bfw_0 + \nabla p\quad\text{with } \bfw_0\in H_0 (\mathrm{div}^0;\Omega _R),\, p\in H^1 (\Omega _R) \, .\label{helm1}
  \end{equation}
Also notice that, in terms of this decomposition, the property $ \bfw|_K \in H^{1/2+s}(\mathrm{div};K) $ for all $K\in\mathcal{T}_h$ means that 
\begin{equation}\label{eq-regprop:p}
 \bfw_0|_K \in H^{1/2+s}(K)^2\, , \,\,
\Delta p|_K = \mathrm{div}\,\bfw|_K \in H^{1/2+s}(K) \qquad \mbox{ for all } K\in\mathcal{T}_h \, .
\end{equation}
We will bound the $L^2(\Omega_R)^2$-norm of $\nabla p$ and also a weaker norm of $\bfw_0$ by means of $\Vert\bfw\Vert _{DG}$ using similar arguments to those
in  \cite{HMP_Maxwell}. In particular, we take the \emph{shape regularity} and \emph{quasi-uniformity}  measures
$$
s.r. (\mathcal{T}_h) = \max _{K\in\mathcal{T}_h}\frac{h_K}{d_K} \quad\text{ and }\quad q.u. (\mathcal{T}_h) =  \max _{K\in\mathcal{T}_h}\frac{h}{h_K}\, ,
$$
where
, for each $K\in\mathcal{T}_h$, we denote by 
$d_K$ the diameter of the largest ball contained in $K$. 

\subsubsection{A bound of $\Vert\nabla p \Vert _{0,\Omega_R}$ by a duality argument}

We consider the adjoint problem of (\ref{prob-FwdProbV_HM}), which consists of finding $\boldsymbol{\phi}\in H(\mathrm{div};\Omega _R)$ such that
\begin{equation}\label{prob-adjFwdProbV_HM}
\begin{array}{l}
 \displaystyle \int_{\Omega _R} \left( \frac{1}{k^2} \, \nabla\cdot\boldsymbol{\phi } \, \nabla\cdot\overline{\bfz} - \boldsymbol{\phi}\cdot\overline{\bfz} \right) d\bfx 
 + \int _{\Sigma _R}  (N_R)^* (\boldsymbol{\phi}\cdot\bfnu) \, (\overline{\bfz}\cdot\bfnu) \, dS \\[1ex]
 \displaystyle \hspace*{6cm} - \int _{\Gamma} {(G^H_{\Gamma})}^* (\boldsymbol{\phi}\cdot\bfnu) \, (\overline{\bfz}\cdot\bfnu) \, dS \, = \, 
 \int_{\Omega _R}  \nabla p\cdot\overline{\bfz} \, d\bfx \, ,
 \end{array}
\end{equation}
for all $\bfz\in H(\mathrm{div},\Omega _R)$. Let us emphasize that (\ref{prob-adjFwdProbV_HM}) is well-posed and shows the following regularity.

\begin{lemma}\label{lemma-adjFwdProbVok}
  For any $p\in H^1 (\Omega _R)$, if $H>0$ is sufficiently small then the adjoint problem (\ref{prob-adjFwdProbV_HM}) is well-posed. Moreover, for each $s\in (0,1/2)$
 the solution has the regularity
 $  \boldsymbol{\phi}\in H^{1/2+s} (\Omega_R)^2$, with
 \[
 \Vert\boldsymbol{\phi}\Vert _{1/2+s,\Omega_R} \leq CH^{-s}\Vert \nabla p \Vert _ {0,\Omega _R}\, , \] where $C>0$ depends only on $\Omega_R$.
\end{lemma}
\begin{proof}
  The well-posedness of the adjoint problem (\ref{prob-adjFwdProbV_HM}) follows from our proof of the well-posedness of the original problem (\ref{prob-FwdProbV_HM}) in Theorem \ref{lemma-E!sol:semidiscrGIBCfwdprob_vec}.

Using the Helmholtz decomposition (\ref{helm1}) and reasoning as in the proof of Lemma \ref{lemma-FwdProbVok}, we see that the solution  of the adjoint problem is the function $\boldsymbol{\phi}=\nabla q$ for $q\in H^1(\Omega _R)$ which solves the following equations in weak sense:
\begin{equation*}
\begin{array}{ll}
 \displaystyle \Delta q+k^2q =-k^2p\quad & \mbox{in }\Omega _R \, ,\\
 \displaystyle - q + (N_R)^* (\frac{\partial q}{\partial \bfnu}) =0& \mbox{on }\Sigma_R  \, ,\\[1ex]
\displaystyle q + {(G^H_{\Gamma})}^* (\frac{\partial q}{\partial\bfnu})=0  & \mbox{on }\Gamma \, .
 \end{array}
\end{equation*}
Thus $q$ can be extended as a solution of a scattering problem to $\Omega$ with the adjoint radiation condition at infinity.  Hence
the regularity of $q$ is determined from the boundary condition on $\Gamma$ and in particular, since $(G^H_{\Gamma})^* (\frac{\partial q}{\partial \boldsymbol{\nu}})\in H^{1+s}(\Gamma)$ for all $0\leq s<1/2$ (see the remark after Assumption \ref{G}), we see that $q\in H^{3/2+s}(\Omega_R)$. Then using an inverse estimate
guaranteed by the assumed quasi-uniformity of the boundary mesh,
\begin{eqnarray*}
\Vert q\Vert_{H^{3/2+s}(\Omega_R)}&\leq& C \Vert
  {(G^H_{\Gamma})}^* (\frac{\partial q}{\partial\bfnu})\Vert_{H^{1+s}(\Gamma)}
  \leq C H^{-s}\Vert
  {(G^H_{\Gamma})}^* (\frac{\partial q}{\partial\bfnu})\Vert_{H^{1}(\Gamma)}
 \\& \leq
 & C H^{-s}\Vert
  \nabla q\cdot\bfnu \Vert_{H^{-1}(\Gamma)}\leq 
 \, C H^{-s}\Vert
  \bfphi\cdot\bfnu \Vert_{H^{-1/2}(\Gamma)}\leq CH^{-s}\Vert\bfphi\Vert_{\Hdiv}.
  \end{eqnarray*}
  Finally, the continuity estimate for the solution of the semidiscrete adjoint problem (\ref{prob-adjFwdProbV_HM}) provides the bound 
  $\Vert\bfphi\Vert_{\Hdiv}\leq C\Vert \nabla p\Vert_{L^2(\Omega_R)}$ and completes the proof.
\end{proof}

Notice that, by the $L^2(\Omega_R)^2$-orthogonality of the Helmholtz decomposition (\ref{helm1}), 
$$
\Vert\nabla p \Vert _{0,\Omega_R} ^2 \, = \, \int _{\Omega _R} \nabla p  \cdot \overline{\nabla p } \, d\bfx  \, = \, \int _{\Omega _R}\nabla p  \cdot \overline{\bfw} \, d\bfx \, .
$$
Making use of the adjoint problem for $\bfz=\bfw$,
$$
\begin{array}{l}
\Vert\nabla p  \Vert _{0,\Omega_R} ^2 
 \, = \,
 \displaystyle \int_{\Omega _R} \left( \frac{1}{k^2} \, \nabla\cdot\boldsymbol{\phi } \, \nabla \cdot \overline{\bfw} - \boldsymbol{\phi}\cdot\overline{\bfw} \right) d\bfx 
 + \int _{\Sigma _R}  (\boldsymbol{\phi}\cdot\bfnu) \, N_R(\overline{\bfw\cdot\bfnu}) \, dS \\
 \hspace*{6cm} \displaystyle
 - \int _{\Gamma}  (\boldsymbol{\phi}\cdot\bfnu) \, {G^H_{\Gamma}} (\overline{\bfw\cdot\bfnu}) \, dS \, .
 \end{array}
$$
If we split the domain $\Omega_R$ in terms of the mesh $\mathcal{T}_h$ and then integrate by parts in each $K\in\mathcal{T}_h$, thanks to Trefftz properties for $\mathbf{w}\in\mathbf{W}^s(\mathcal{T}_h)$, we come up with
$$
\begin{array}{l}
\Vert\nabla p \Vert _{0,\Omega_R} ^2 
 \, = \,
 \displaystyle
 \int _{\mathcal{F}_h^I}\frac{1}{\sqrt{\alpha _2 \, k^3}}\, \boldsymbol{\phi} \cdot \bfnu \,\frac{\sqrt{\alpha _2}}{\sqrt{k}} \, \jmp{  \nabla\cdot \overline{\bfw} }_{\bfnu}  \, dS \\
 \hspace*{4cm} \displaystyle
 + \int _{\Sigma _R} \frac{\sqrt{k}}{\sqrt{\delta } k^2} (\boldsymbol{\phi}\cdot\bfnu) \, \, \frac{\sqrt{\delta}}{\sqrt{k} } \left(  \overline{ \nabla \cdot \bfw  + k^2\,  N_R(\bfw \cdot\bfnu) } \right) dS \\
 \hspace*{4cm} \displaystyle
 - \int _{\Gamma} \frac{\sqrt{ k }}{\sqrt{\tau} \, k^2 }\, (\boldsymbol{\phi}\cdot\bfnu) \, \frac{\sqrt{\tau} }{\sqrt{ k }} \left(\overline{  (  \nabla\cdot  {\bfw}  + k^2\, G^H_{\Gamma} ( {\bfw}\cdot\bfnu) } \right) dS \, .
 \end{array}
$$
Then, by Cauchy-Schwartz inequality and the lower bound of the DG-norm (\ref{DGnormbound}), 
\begin{equation}\label{eq-boundw0_1}
\Vert\nabla p \Vert _{0,\Omega_R} ^2 
 \, \leq \,
 (\mathcal{G}(\boldsymbol{\phi}) ) ^{1/2} \, \Vert \bfw\Vert _{DG} \, ,
\end{equation}
where we denote
$$
\begin{array}{l}
\mathcal{G}(\boldsymbol{\phi}) 
 \, = \,
 \displaystyle \frac{1}{k^3} \big(
  \int_{\mathcal{F}_h^I} \frac{1}{\alpha _2 }\, |\boldsymbol{\phi}\cdot\bfnu|^2 \, dS 
 + \int _{\Sigma _R} \frac{1}{\delta } \, |\boldsymbol{\phi}\cdot\bfnu|^2 \, dS 
 + \int _{\Gamma} \frac{1}{\tau  }\, |\boldsymbol{\phi}\cdot\bfnu|^2 \, dS \big)
  \, .
 \end{array}
$$
In order to deal with this last term, we use the following trace inequality (see \cite[eq. (24)]{HMPhp}):
$$
\Vert \eta \Vert ^2_{0,\partial K} \,\leq\, C \left( \frac{1}{h_K} \, \Vert \eta\Vert _{0,K}^2 + h_K^{2s}\, |\eta|_{1/2+s,K}^2\right)\quad \forall \eta\in H^{1/2+s}( K), K\in\mathcal{T}_h;
$$
indeed, taking $\eta $ to be each entry of $\boldsymbol{\phi} \in H^{1/2+s} (\Omega _R) ^2$, we deduce
$$
\begin{array}{l}
\mathcal{G}(\boldsymbol{\phi}) 
 \, \leq \,
 \displaystyle \frac{1}{k^3 \min \{ \underline{\alpha }_2 , \underline{\delta }  , \underline{ \tau } \}  }
 \sum _{K\in\mathcal{T}_h }
  \int_{\partial K }  |\boldsymbol{\phi}\cdot\bfnu|^2 \, dS \\
 \hspace*{3cm} \displaystyle
 \leq 
 \displaystyle \frac{C}{k^3\min \{ \underline{\alpha }_2, \underline{\delta }  ,  \underline{ \tau } \}  }
 \sum _{K\in\mathcal{T}_h }
   \left( \frac{1}{h_K} \, \Vert \boldsymbol{\phi} \Vert _{0, K}^2 + \frac{h_K^{2s}}{H^{2s}}\, |\boldsymbol{\phi}|_{1/2+s,K}^2\right) 
 \, ,
 \end{array}
$$
where $\underline{\alpha }_2 = \mathrm{inf} _{\mathcal{F}_h^I } \alpha _2$, $\underline{\delta } = \mathrm{inf} _{\mathcal{F}_h^R } \delta $ and $\underline{\tau } = \mathrm{inf} _{\mathcal{F}_h^{\Gamma} } \tau $ (we have assumed  $\underline{\alpha }_2 , \underline{\delta },\underline{\tau} >0$). Recalling that $q.u. (\mathcal{T}_h) \, h\leq h_K \leq h$ and applying Lemma \ref{lemma-adjFwdProbVok}, we have %
$$
\begin{array}{l}
\mathcal{G}(\boldsymbol{\phi}) 
 \, 
 \leq \,
 \displaystyle \frac{C}{k^3 \min \{ \underline{\alpha }_2 , \underline{\delta }  , \underline{ \tau } \}  }
 \sum _{K\in\mathcal{T}_h }
   \left( \frac{1}{q.u. (\mathcal{T}_h) \, h } \, \Vert \boldsymbol{\phi} \Vert _{0, K}^2 + h^{2s}H^{-2s}\, |\boldsymbol{\phi}|_{1/2+s,K}^2\right)  \\[2ex]
 \hspace*{2cm} \displaystyle
 \, 
 \leq \,
  \frac{C}{k^3\min \{ \underline{\alpha }_2 , \underline{\delta } , \underline{ \tau } \}  }\,
   \Big( \frac{1}{q.u. (\mathcal{T}_h) \, h } \, \Vert \boldsymbol{\phi} \Vert _{0,\Omega _R }^2 + h^{2s}H^{-2s}\, |\boldsymbol{\phi}|_{1/2+s,\Omega _R}^2\Big) 
   \\[2ex]
 \hspace*{2cm} \displaystyle
 \,\leq\,  \frac{C}{\min \{ \underline{\alpha }_2 , \underline{\delta }  , \underline{ \tau } \}  }\,
 \Big( \frac{1}{q.u. (\mathcal{T}_h) \, h }  + h^{2s}H^{-2s}\Big)\, H^{-2s}\,  \Vert \nabla p \Vert _ {0,\Omega _R} ^2  
 \, .
 \end{array}
$$
Therefore, using (\ref{eq-boundw0_1}),
$$
\Vert \nabla p \Vert _{0 ,\Omega_R}
 \, \leq \,
 \left( \frac{C\, H^{-2s} }{\min \{ \underline{\alpha }_2, \underline{\delta }  , \underline{ \tau } \}  } \, \Big( \frac{1}{q.u. (\mathcal{T}_h) \, h }  + h^{2s}\, H^{-2s} \Big)  \right) ^{\! 1/2} \, \Vert \bfw\Vert _{DG} \, ,
$$
so we have proved the following lemma:
\begin{lemma}\label{lem1}  For sufficiently small $H>0$ and any $s\in (0,\infty)$, there is a constant $C$ (depending on $s$ but independent of $h$ and $H$), such that
$$
\Vert \nabla p \Vert _{0,\Omega_R} 
 \, \leq \, 
 CH^{-s} ( h^{-1/2} + h^{s}\, H^{-s} ) \, \Vert \mathbf{w}\Vert _{DG} \, ,
$$
where $\mathbf{w}\in \mathbf{W}^s(\mathcal{T}_h)$ and $p\in H^1 (\Omega_R)$  satisfies (\ref{helm1}).
\end{lemma}

\subsubsection{A bound of $\Vert\mathbf{w} _0 \Vert _{ H(\mathrm{curl};\Omega_R ) '}$ by a duality argument}

For any trial function $\mathbf{u}\in H(\mathrm{{\rm curl}};\Omega_R)$ we consider its $L^2 (\Omega _R)^2$-orthogonal Helmholtz decomposition as in (\ref{helm1}):
\begin{equation}\label{helm1bis}
\mathbf{u} = \mathbf{u}_0 + \nabla q\quad\text{with } \mathbf{u}_0\in H_0(\mathrm{div}^0;\Omega _R),\, q\in H^1 (\Omega _R) \, .
\end{equation}
Then, using the $L^2 (\Omega _R)$-orthogonality of Helmholtz decomposition as well as Trefftz property,
$$
\int_{\Omega _R }\bfw_0 \cdot \mathbf{u} \, d\bfx \, = \, 
\int_{\Omega _R }\bfw_0 \cdot \mathbf{u}_0 \, d\bfx \, = \, 
- \frac{1}{k^2 } \, \sum _{K\in\mathcal{T}_h }
  \int_{ K } \nabla \nabla\cdot \bfw_0 \cdot \mathbf{u} _0 \, d\bfx \, ,
$$
so that, integrating by parts, and using that $\mathbf{u}_0\in H_0(\mathrm{div}^0;\Omega_R)$,
$$
\begin{array}{l}
\displaystyle\int_{\Omega _R }\bfw_0 \cdot \mathbf{u} \, d\bfx \, = \, 
- \frac{1}{k^2 } \, \sum _{K\in\mathcal{T}_h }
  \int_{ \partial K } \nabla\cdot \bfw _0 \, \mathbf{u}_0 \cdot\bfnu \, dS \, = 
  \\ \hspace*{1cm}
= \, - \displaystyle \frac{1}{k^2 } \left( 
\int _{\mathcal{F}^I_h }  \jmp{ \nabla\cdot \mathbf{w}_0}  _{ \bfnu  } \, \mathbf{u}_0 \, dS 
+ \int _{\mathcal{F}^R_h } \nabla\cdot \bfw_0 \, \bfnu  \cdot \mathbf{u}_0  \, dS 
- \int _{\mathcal{F}^{\Gamma}_h } \nabla\cdot \bfw_0 \, \bfnu  \cdot \mathbf{u}_0  \, dS 
\right) \, = 
\\ \hspace*{1cm}
= - \displaystyle \frac{1}{k^2 } \,  
\int _{\mathcal{F}^I_h } \jmp{ \nabla\cdot \bfw _0 } _{ \bfnu  } \bfnu  \, \mathbf{u} _0 \cdot \bfnu  \, dS \, .  
  \end{array}
$$
Therefore, by Cauchy-Schwartz inequality,
$$
|\int_{\Omega _R }\bfw _0 \cdot \mathbf{u}  \, d\bfx | \, \leq \, 
 \displaystyle \frac{1}{k^2 } \,  
 \Big( \int _{\mathcal{F}^I_h } \frac{\alpha _2}{k} \, |\jmp{ \nabla\cdot \bfw _0 } _{ \bfnu  } |^2 \, dS \Big) ^{1/2}
\,
 \Big( \int _{\mathcal{F}^I_h } \frac{k}{\alpha _2 \, } \, |\mathbf{u} _0 \cdot \bfnu |^2 \, dS 
\Big) ^{1/2}
\, ,
$$
and we have
$$
|\int_{\Omega _R }\bfw _0 \cdot\mathbf{u}  \, d\bfx | \, \leq \, 
 \displaystyle \frac{1}{\sqrt{ k^3\, \underline{\alpha}_2 }} \,  
 \Vert \bfw \Vert _{DG  } 
\,
 \Big( \sum _{K\in\mathcal{T}_h } \int_{ \partial K }  | \mathbf{u}_0 \cdot \bfnu  |^2 \, dS 
\Big) ^{1/2}
\, .
$$
Let us notice that, taking the $\mathrm{curl}$ of (\ref{helm1bis}), we know that $\mathrm{curl}\mathbf{u}_0=\mathrm{curl}\mathbf{u}\in L^2(\Omega_R)$; this allows us to use again the trace inequality \cite[eq. (24)]{HMPhp} for $\mathbf{u}_0\in H_0(\mathrm{div}^0;\Omega _R)\cap H (\mathrm{curl};\Omega _R)\hookrightarrow H^{1/2+s} (\Omega _R)$ for any $s\in [0,1/2)$:
$$
\Vert \mathbf{u}_0 \cdot \bfnu  \Vert ^2_{0,\partial K} \,\leq \, \Vert \mathbf{u}_0  \Vert ^2_{0,\partial K} \,\leq\, C \left( \frac{1}{q.u.(\mathcal{T}_h ) \, h} \, \Vert \mathbf{u}_0\Vert _{0,K}^2 + h^{2s}\, | \mathbf{u}_0 |_{1/2+s,K}^2\right) .
$$
Then, summing over $K\in\mathcal{T}_h$ and making use of the continuity of the embedding $H_0(\mathrm{div}^0;\Omega _R)\cap H (\mathrm{curl};\Omega _R)\hookrightarrow H^{1/2+s} (\Omega _R)$, we deduce
$$
\begin{array}{l}
 \displaystyle\sum _{K\in\mathcal{T}_h } \Vert \mathbf{u}_0\cdot \bfnu  \Vert ^2_{0,\partial K} 
 \,\leq\, C \left( \frac{1}{q.u.(\mathcal{T}_h ) \, h} \, \Vert  \mathbf{u}_0\Vert _{0,\Omega _R }^2 + h^{2s}\, | \mathbf{u}_0 |_{1/2+s,\Omega _R }^2\right) \,\leq\\
\hspace*{2cm} \displaystyle \,\leq\, C \left( \frac{1}{q.u.(\mathcal{T}_h ) \, h}  + h^{2s} \right) (\Vert \mathbf{u}_0 \Vert _{0,\Omega _R }^2 + \Vert \mathrm{curl} \mathbf{u}_0 \Vert _{0,\Omega _R }^2 ) \, .
 \end{array}
$$
But recalling the $L^2(\Omega_R)^2$-orthogonality of Helmholtz decomposition (\ref{helm1bis}), 
$$
\Vert \mathbf{u} \Vert _{\mathrm{curl},\Omega_R} ^2 \, = \, 
\Vert \mathbf{u}_0 \Vert _{0,\Omega_R} ^2 + 
\Vert \nabla q \Vert _{0,\Omega_R} ^2 +
\Vert \mathrm{curl}\mathbf{u}_0 \Vert _{0,\Omega_R} ^2
\geq  
\Vert \mathbf{u}_0 \Vert _{0,\Omega_R} ^2 + 
\Vert \mathrm{curl}\mathbf{u}_0 \Vert _{0,\Omega_R} ^2
\, ,
$$
so that
$$
 \displaystyle\sum _{K\in\mathcal{T}_h } \Vert \mathbf{u}_0\cdot\bfnu \Vert ^2_{0,\partial K} 
\,\leq\, C \left( \frac{1}{q.u.(\mathcal{T}_h ) \, h}  + h^{2s} \right) \Vert \mathbf{u}\Vert _{\mathrm{curl},\Omega_R} ^2 \, .
$$
Therefore
$$
|\int_{\Omega _R }\bfw _0 \cdot \mathbf{u} \, d\bfx | \, \leq \, 
 \frac{C \, ( ( q.u.(\mathcal{T}_h ) \, h )^{-1} + h^{2s} ) ^{1/2} }{\sqrt{k^3\, \underline{\alpha}_2}} \,  
 \Vert \bfw \Vert _{DG  } 
\, 
\Vert \mathbf{u}\Vert _{\mathrm{curl},\Omega_R}  \, ,
$$
and we conclude that
$$
\Vert \bfw _0  \Vert _{ H(\mathrm{curl};\Omega _R)' }  \, \leq \, 
 \frac{C \, ( ( q.u.(\mathcal{T}_h ) \, h )^{-1} + h^{2s} ) ^{1/2} }{\sqrt{k^3 \, \underline{\alpha}_2} } \,  
 \Vert \bfw \Vert _{DG  }  \, .
$$
We have proved the following lemma:
\begin{lemma}\label{lem2} For each $s\in (0,\infty)$, there exists a constant $C$ (depending on $s$ but independent of $h>0$ and $H>0$ small enough) such that
$$
\Vert \mathbf{w} _0  \Vert _{ H(\mathrm{curl};\Omega _R)' }  \, \leq \, C h^{-1/2} \Vert \mathbf{w} \Vert _{DG  }  \, ,
$$
for all $\mathbf{w}\in\mathbf{W}^s (\mathcal{T}_h)$ and $\mathbf{w}_0\in H_0(\mathrm{div}^0;\Omega_R)$ satisfying (\ref{helm1}).
\end{lemma}

Using Lemmas ~\ref{lem1} and \ref{lem2} we have then proved the following theorem that summarizes our results from this section.
 \begin{theorem}\label{th-conv_Hh2H} For all sufficiently small $H$ and for each $s>0$, there is a constant $C$ depending on $s$ but independent of $h$, $H$, $\mathbf{v}^H$ and $\mathbf{v}_h^H$ such that
   \[
   \Vert \mathbf{v}^H-\mathbf{v}^H_h\Vert_{H(\mathrm{curl},\Omega_R)'}\,\leq\,  C\,H^{-s}\, (h^{-1/2}+ h^{s}\, H^{-s})\, \inf_{\mathbf{w}_h\in\mathbf{W}_h} \Vert \mathbf{v}-\mathbf{w}_h\Vert_{DG^+} \, .
   \]
 \end{theorem}
 \begin{remark} We could avoid the growing factor $H^{-s}$ if we used $\mathcal{C}^1$  
functions to compute $G_\Gamma^H$. Note that using a trigonometric basis is one case where $H^{-s}$ can be removed. In the general case, the $H^{-s}$ terms are not unexpected. Also notice that we can choose a fine mesh on the boundary $\Gamma$ and then we still can approximate the scattering problem with sufficient accuracy. 

\begin{proof}
 Let $\bfw=\bfv^H-\bfv_h^H$ in (\ref{helm1}). Then
    \[
    \Vert \bfw^H-\bfv_h^H\Vert_{H(\mathrm{curl};\Omega_R)'}\leq
    \Vert \bfw_0\Vert_{H(\mathrm{curl};\Omega_R)'}+
    \Vert \nabla p\Vert_{H(\mathrm{curl};\Omega_R)'}.
    \]
  Use of Lemmas \ref{lem1} and \ref{lem2} completes the proof.
\end{proof}

Using approximation results from \cite{HMP_approx} this theorem could
be converted into order estimates.  Because of the poor regularity of
$\mathbf{v}^H$ (due to the reduced regularity imposed by the
regularity of $G_\Gamma^H$) we expect to need a refined $h$ grid near
the boundary unless we use smooth basis functions on $\Gamma$.
 \end{remark}
 
Our final result combines Theorem
\ref{lemma-E!sol:semidiscrGIBCfwdprob_vec} and Theorem
\ref{th-conv_Hh2H} to give an error estimate for the fully discrete
problem.
\begin{corollary} Under the conditions of Theorem~\ref{th-conv_Hh2H}, there is a constant $C$ (depending on $s\in (0,\infty)$ but independent of $h$, and $H$ small enough, $\mathbf{v}^H$ and $\mathbf{v}_h^H$) such that
   \[
\begin{array}{rl}
\displaystyle  \Vert \mathbf{v}-\mathbf{v}^H_h\Vert_{H(\mathrm{curl};\Omega _R)'}\, \leq\,  &
\displaystyle C H^{-s}\, (h^{-1/2}+h^{s}H^{-s}) \, \Big( \inf_{\mathbf{w}_h\in\mathbf{W}_h} \Vert \mathbf{v}-\mathbf{w}_h\Vert_{DG^+}+\Big. \\[1ex]
& \displaystyle \qquad \Big. \Vert (G_\Gamma^H-G_\Gamma)(\mathbf{v}\!\cdot\!\bfnu)\Vert_{H^{1/2}(\Gamma)}
   +\Vert (G_\Gamma^H-G_\Gamma)g\Vert_{H^{1/2}(\Gamma)}\Big) .
\end{array}
   \]
 \end{corollary}
 
 \begin{proof}
We have already shown in {Theorem \ref{lemma-E!sol:semidiscrGIBCfwdprob_vec}} 
that $\bfv^H$ converges to $\bfv$ in $\Hdiv$, and this is a stronger norm than $H(\mathrm{curl}; \Omega_R)'$. Besides, we have shown in Theorem \ref{th-conv_Hh2H} that $\bfv^H_h$ converges to $\bfv^H$ in $H(\mathrm{curl}; \Omega_R)'$. Combining those two results, we conclude the statement.
   \end{proof}

 \section{Numerical Results}\label{num}
%
We now present some limited and preliminary numerical results that illustrate the foregoing error analysis.  We consider scattering by
a circular disk and use absorbing boundary conditions (ABCs) on the exterior.  This is not exactly what is analyzed in the foregoing sections, \pmr{but ABCs give rise to GIBC and are an important application of our method. Now there is no Neumann-to-Dirichlet
map and the GIBC is provided by various Absorbing Boundary Conditions. However, the same error estimates hold in this case, and the use of higher order ABCs is of considerable practical use. Our results are obtained using a modified version of LehrFEM~\cite{lehrfem}.  We make the choice $\alpha_1=\alpha_2=\tau=1/2$ so that we are actually using the 
Ultra Weak Variational formulation.  The wave-number  is $k=8$ and we use a uniform number of directions on the elements
with $p_K=7$ for all elements $K$.}

In particular, we consider scattering of a plane wave incident field $u^i(r,\theta) = \exp (ikr\cos\theta)$ from a circular scatterer of radius $a$.
We impose the Dirichlet boundary condition $u(a,\theta)=-\exp (ika\cos\theta)$ on the boundary of the scatterer, and ABCs 
of the form:
 \[
  \alpha u + \beta \Delta_{1}u  = -\frac{\partial u}{\partial r}\;\;\mbox{on}\;\;\Sigma_{\!_R}.
 \]
where $\Delta_{1} =\frac{1}{R^2}\partial^2_\theta$ is the Laplace-Beltrami operator on the circle $\Sigma_{\!_R}$.  This is obviously a special case of the GIBC (second equation in (\ref{vbasic})) albeit on the outer rather than inner boundary.

These boundary conditions are given in \cite{kangfeng} as approximations of the Dirichlet-to-Neumann map, for the following choices of $\alpha$ and $\beta$:
\begin{equation*}
\begin{aligned}
 & \mbox{ABC0:}\;\;\alpha = ik,\;\;\beta = 0,\;\;\\
 & \mbox{ABC1:}\;\;\alpha = \left(ik+\frac{1}{2R}\right),\;\;\beta = 0,\\
  & \mbox{ABC2:}\;\;\alpha = \left(ik+\frac{1}{2R}+\frac{i}{8kR^2}\right),\;\;\beta = \frac{i}{2kR^2},\\
  & \mbox{ABC3:}\;\;\alpha = \left(ik+\frac{1}{2R}+\frac{i}{8kR^2}-\frac{1}{8k^2R^3}\right),\;\;\beta = \left(\frac{i}{2kR^2}-\frac{1}{2k^2R^3}\right).
 \end{aligned}
\end{equation*}
\pmr{The above ABCs can then be translated into GIBCs for the displacement problem as outlined in Section 1, and the 
results can then be translated back to the field $u$ using the Trefftz property so that $u=-\nabla\cdot\bfv/k^2$.}

An advantage of the simple circular scatterer is that we can write down the  exact solution for scattering of a plane wave incident field with an ABC on $\Sigma_{\!_R}$. In particular: 
\[
 u(r,\theta)=\sum_{n=-\infty}^\infty\left[ a_n H_n^{(1)}(kr) + b_n H_n^{(2)}(kr)\right]\exp(in\theta),
\]
where the coefficients $a_n, b_n$ are solutions of the linear system:
\[ \left( \begin{array}{ccc}
\gamma H_n^{(1)}(kR)+kH_n^{(1)'}(kR) &\hspace{.2cm}  & \gamma H_n^{(2)}(kR)+kH_n^{(2)'}(kR)\\
 &  & \\
H_n^{(1)}(ka)&  & H_n^{(2)}(ka)\end{array}\right)\left(\begin{array}{c}a_n \\ b_n\end{array}\right)=\left(\begin{array}{c}0 \\ -i^nJ_n(ka)\end{array}\right)\] 
and $\gamma  = \alpha - \beta n^2/R^2$.  \pmr{This then gives a series solution for $\bfv$ that solves (\ref{vbasic}).}

In Fig.~\ref{fig:abcntd1} we show field plots for the solution of the approximate scattering problem in 7 cases, together with the exact solution of the full scattering problem. For this experiment we choose $k=8$ giving a wavelength
$\lambda\approx 0.78$.  The mesh size requested from the mesh generator is $h=0.1$.  The outer boundary is distance 0.64 wavelengths from the scatterer which is rather close. Indeed the top left and top right panels show considerable distortion compared to the lower right figure showing the exact solution.  Using ABC2 and ABC3 gives better fidelity.  In the bottom left panel we
show the solution computed using a discretized NtD map, which shows the best accuracy.

\begin{figure}
\begin{center}
\resizebox{0.40\textwidth}{!}{\includegraphics{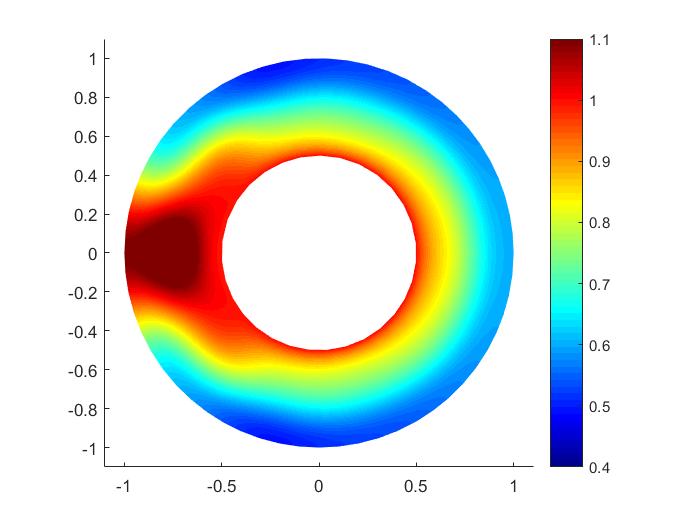}}
\resizebox{0.40\textwidth}{!}{\includegraphics{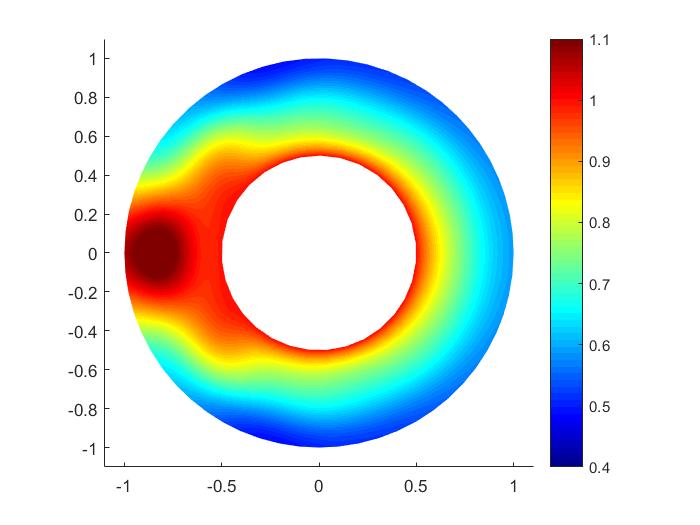}}
\resizebox{0.40\textwidth}{!}{\includegraphics{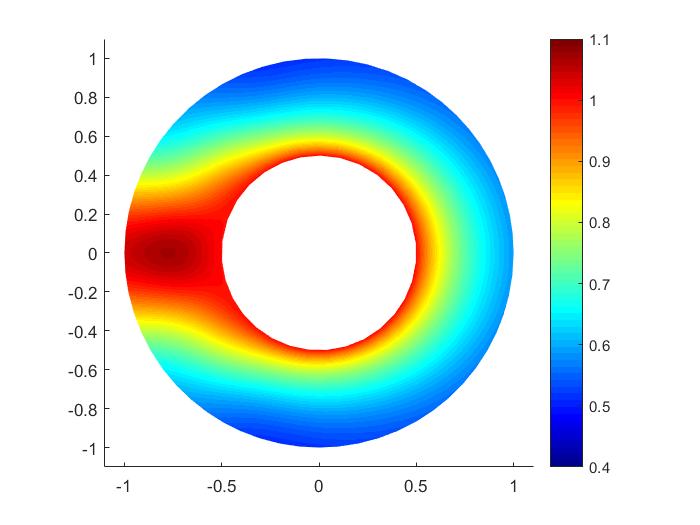}}
\resizebox{0.40\textwidth}{!}{\includegraphics{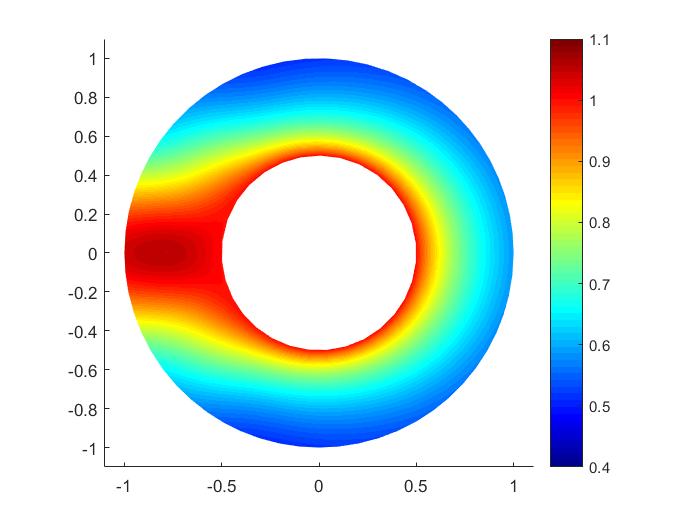}}
\resizebox{0.40\textwidth}{!}{\includegraphics{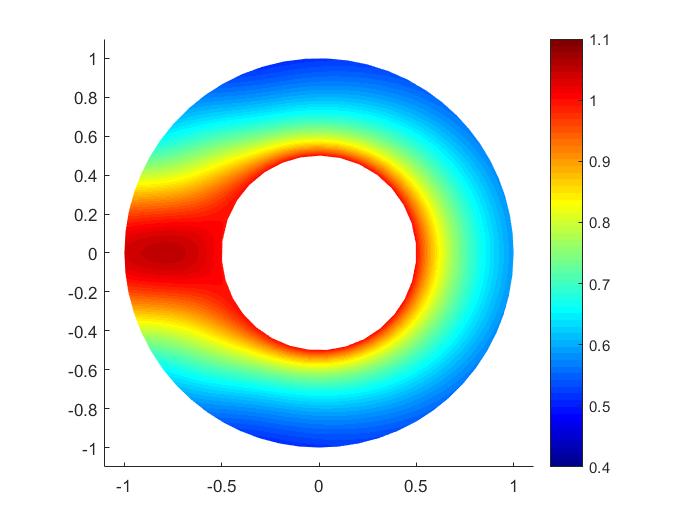}}
\resizebox{0.40\textwidth}{!}{\includegraphics{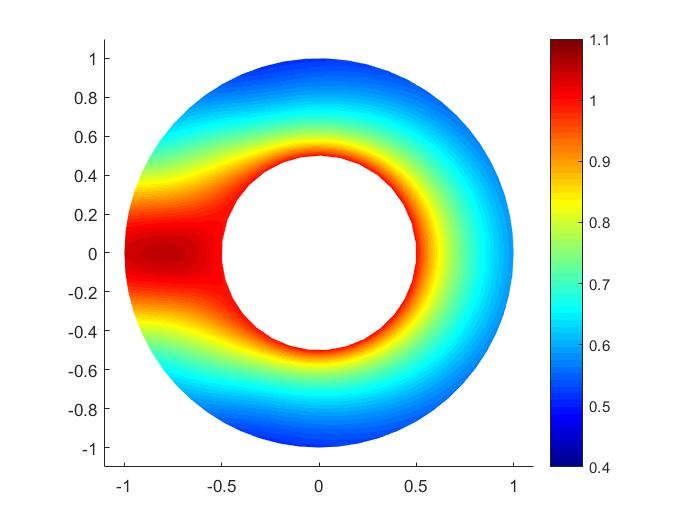}}
\caption{Moduli of the fields scattered from a disk: $a=0.5$, $R=1$, $p=7$ plane waves per element, $k=8$. Top left: ABC0. Top right: ABC1. Middle left: ABC2. Middle right: ABC3. Bottom left: NtD. Bottom right: Exact solution. For these results $h=0.1$ is fixed and 13 Fourier modes are used for the ABC and NtD models.  The exact solution uses 40 modes.}
\label{fig:abcntd1}
\end{center}
\end{figure}

To investigate convergence in a more quantitative way, in Fig.~\ref{fig:abcntd2} we show the relative $L^2$ error 
on the domain as a function of $1/h$.  In the left hand panel we verify that no matter which ABC is used, the UWVF solution with an ABC or the NtD boundary condition converges optimally to the exact solution for the
particular ABC (essentially the result from Theorem \ref{TheoremQuasiOptimal}).  Then in the right hand panel we compare the UWVF with ABC boundary condition to the true exact solution of the full scattering problem.  As can be seen ABC0 and ABC1 result in a poor relative error and do not benefit from mesh refinement (the absorbing boundary is too close to the scatterer and all error is related to the ABC not the UWVF).  For ABC2 and ABC3 the solution does converge with $h$ until the error from the ABC dominates.  It is clear that in this case ABC3 can be used to obtain a solution with better
than 1\% error even with the close absorbing boundary.  The discrete NtD solution continues to converge for all $h$ in our study and would be preferred in this case (but ABCs can be used on non-circular absorbing boundaries and so are of practical interest).
\begin{figure}
\begin{center}
\resizebox{0.425\textwidth}{!}{\includegraphics{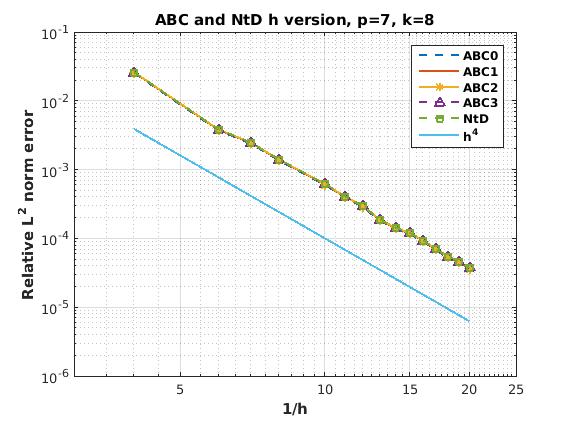}}
\resizebox{0.40\textwidth}{!}{\includegraphics{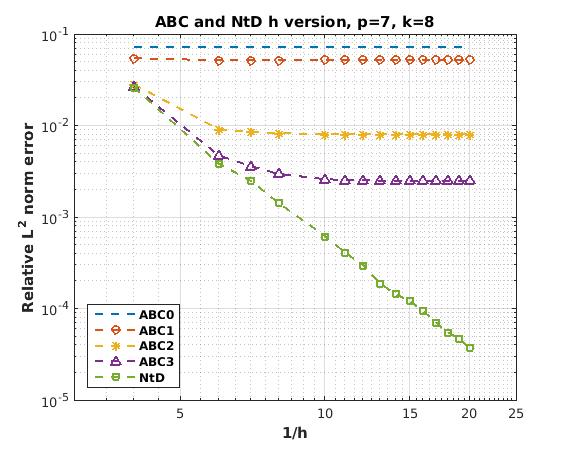}}
\caption{Relative $L^2(\Omega_R)$-norm error vs. $1/h$, for different ABCs on $\Sigma_{\!_R}$.
Left: Errors are computed against the exact solution for the particular ABCs. Right: Errors are computed against the exact solution of the scattering problem.}
\label{fig:abcntd2}
\end{center}
\end{figure}

\section{Conclusion}\label{concl}
We have provided an error analysis of the UWVF discretization of the Helmholtz equation in the presence of a Generalized Impedance Boundary Conditions.  The error analysis is backed by limited numerical experiments.

Clearly several extensions and further numerical tests need to be performed.  In particular our analysis and numerical tests are
for a smooth boundary.  Analysis for a non-smooth boundary, and appropriate mesh refinement strategies near corners
need to be developed.

\section*{Acknowledgements}  The research of Peter Monk is partially supported by NSF grant number DMS-1619904. The research of Virginia Selgas is partially supported by MCI project number MTM2013-43671-P. Peter Monk and Shelvean Kapita acknowledge 
the support of the IMA, University of Minnesota during the special
year ``Mathematics and Optics".

\end{document}